\documentclass[10pt]{amsart}
\usepackage[leqno]{amsmath}
\usepackage{amssymb,latexsym,soul,cite,amsthm,color,enumitem,graphicx,mathtools,microtype,accents, esint, tikz, comment} 
\usepackage[colorlinks=true,
            linkcolor=blue,
            citecolor=blue,
            urlcolor=blue]{hyperref}
\definecolor{cobalt}{rgb}{0.0, 0.28, 0.67}
\usepackage[english]{babel}
\usepackage[left=2.5cm,right=2.5cm,top=2.5cm,bottom=2.5cm]{geometry}

\allowdisplaybreaks
\makeatletter
\DeclareRobustCommand*{\bfseries}{%
  \not@math@alphabet\bfseries\mathbf
  \fontseries\bfdefault\selectfont
  \boldmath
}

\numberwithin{equation}{section}

\newtheorem{theorem}{Theorem}[section]
\theoremstyle{plain}
\newtheorem{lemma}[theorem]{Lemma}
\theoremstyle{plain}
\newtheorem{proposition}[theorem]{Proposition}
\theoremstyle{plain}

\newtheorem{definition}[theorem]{Definition}
\theoremstyle{definition}

\newcommand{\N}{{\mathbb N}}

\newcommand{\R}{{\mathbb R}}
\newcommand{\eps}{\varepsilon}
\newcommand{\mm}{\mathtt{m}}
\newcommand{\rr}{\mathtt{r}}
\renewcommand{\epsilon}{\varepsilon}
\renewcommand{\rho}{\varrho}
\newcommand{\beq}{\begin{equation}}
\newcommand{\eeq}{\end{equation}}
\renewcommand{\le}{\leqslant}
\renewcommand{\ge}{\geqslant}

\renewcommand{\d}{\mathrm{d}}                 
\newcommand{\dx}{\mathrm{d}x}
\newcommand{\dy}{\mathrm{d}y}
\newcommand{\dz}{\mathrm{d}z}
\newcommand{\dt}{\mathrm{d}t}

\newcommand{\loc}{\mathrm{loc}}
\DeclareMathOperator*{\esssup}{ess\,sup}

\def\Xint#1{\mathchoice
{\XXint\displaystyle\textstyle{#1}}%
{\XXint\textstyle\scriptstyle{#1}}%
{\XXint\scriptstyle\scriptscriptstyle{#1}}%
{\XXint\scriptscriptstyle\scriptscriptstyle{#1}}%
\!\int}
\def\XXint#1#2#3{{\setbox0=\hbox{$#1{#2#3}{\int}$ }
\vcenter{\hbox{$#2#3$ }}\kern-.6\wd0}}
\def\fint{\Xint-}

\def\Yint#1{\mathchoice
    {\YYint\displaystyle\textstyle{#1}}%
    {\YYint\textstyle\scriptstyle{#1}}%
    {\YYint\scriptstyle\scriptscriptstyle{#1}}%
    {\YYint\scriptscriptstyle\scriptscriptstyle{#1}}%
      \!\iint}
\def\YYint#1#2#3{{\setbox0=\hbox{$#1{#2#3}{\iint}$}
    \vcenter{\hbox{$#2#3$}}\kern-.51\wd0}}
\def\longdash{{-}\mkern-3.5mu{-}} 

\def\fiint{\Yint\longdash}

\makeatletter
\newcommand{\leqnomode}{\tagsleft@true}
\newcommand{\reqnomode}{\tagsleft@false}
\makeatother

\newenvironment{enumroman}{\begin{enumerate}

}{\end{enumerate}}
\usepackage{graphicx} 

\title[Integral Harnack inequality  and expansion of positivity for f.p.m. equations]{Integral Harnack inequality and expansion of positivity for fractional Porous Medium equation}
\author[Cassanello]{Filippo Maria Cassanello}  \address{Filippo Maria Cassanello\\Università degli Studi di Cagliari, Via Ospedale 72, 09124, Cagliari, Italy }
\email{\url{filippom.cassanello@unica.it}}
\author[De Filippis]{Filomena De Filippis}  \address{Filomena De Filippis\\Fachbereich Mathematik, Universität Salzburg, Hel lbrunner Str. 34, 5020 Salzburg, Austria}
\email{\url{filomena.defilippis@plus.ac.at}}
\author[Stanko]{Calvin Stanko}  \address{Calvin Stanko\\Fachbereich Mathematik, Universität Salzburg, Hel lbrunner Str. 34, 5020 Salzburg, Austria }
\email{\url{calvin.stanko@plus.ac.at}}

\begin{document}

\subjclass[2020]{\vspace{1mm} 35B65, 35R11, 76S05} 

\keywords{\vspace{1mm} Integral Harnack inequality, expansion of positivity, porous medium, nonlocal}

\thanks{{\it Acknowledgements.} This research was funded in whole or in part by the Austrian Science Fund (FWF) [10.55776/PAT1850524]. For open access purposes, the author has applied a CC BY public copyright license to any author accepted manuscript version arising from this submission. \\ The second author is also supported by INdAM - GNAMPA Project, CUP E53C25002010001.}

\begin{abstract}
We establish an integral Harnack inequality and  expansion of positivity for solutions to fractional porous medium-type equations in the fast diffusion regime, under optimal tail assumptions.

\end{abstract}

\maketitle
\begin{center}
\begin{minipage}{12cm}
\tableofcontents
\end{minipage}
\end{center}
\section{Introduction}
\noindent
We are interested in the study of integral Harnack-type estimates and expansion of positivity for local weak solutions to the fractional porous medium equation
\begin{equation}\label{eq:PDE}
    \partial_t u^q-\mathcal{L}_K^s u=0
    \qquad \mbox{in } \Omega_T,
\end{equation}
where $s\in(0,1)$, $q>1$, and $\Omega_T:=\Omega\times(0,T]$ is a space-time cylinder over an open and bounded set $\Omega\subset\mathbb R^N$, $N\geq2$. The nonlocal operator $\mathcal{L}_K^s$ is formally defined by
$$
    -\mathcal{L}_K^s u(x,t)
    =2\,\mathrm{P.V.}\int_{\mathbb R^N}(u(x,t)-u(y,t))K(x,y,t)\,\dy,
$$
where $\mathrm{P.V.}$ denotes the principal value and the kernel $K:\mathbb R^N\times\mathbb R^N\times(0,T]\to[0,\infty)$ is measurable and satisfies
\begin{equation}\label{k}
    \frac{\texttt{L}^{-1}}{|x-y|^{N+2s}}
    \leq K(x,y,t)=K(y,x,t)
    \leq \frac{\texttt{L}}{|x-y|^{N+2s}},
\end{equation}
for every $t\in(0,T]$, $(x,y)\in\mathbb R^N\times\mathbb R^N$, and for some constant $\texttt{L}\geq1$.

The local counterpart belongs to the broader class of doubly nonlinear
parabolic equations
\[
    \partial_t u^q-\operatorname{div}\bigl(|Du|^{p-2}Du\bigr)=0,
\]
which provide classical models of nonlinear diffusion. Their
mathematical interest lies in the interaction between the nonlinear time
derivative and, when $p\neq 2$, the nonlinear diffusion. The local regularity theory for this class has been extensively developed; we refer to \cite{BDL_21,BDKS20,BDLS23,BDMS18,BDMSV18, BDGLS, KK07,KSU12} and the references therein for a comprehensive account. The first investigations of boundedness and gradient regularity go back to the works of Ivanov \cite{I4,I0,I2,I1}, see also \cite{HP85}. More recently, B\"ogelein \& Duzaar \& Gianazza \& Liao \& Scheven \cite{BDGLS} established H\"older continuity of the gradient of weak solutions in the fast diffusion regime. A key ingredient in their argument is a time-insensitive Harnack inequality, which, for $p=2$, is valid in the range
\begin{equation} \label{onqq}
     1<q<\frac{N}{(N-2)_{+}},
\end{equation}
see also \cite{DKV91, DK92}. Its proof combines an integral Harnack inequality, which in turn relies on the quantitative local boundedness estimate available in the larger range
\begin{equation}\label{bound}
    1<q<\frac{N+2}{(N-2)_{+}},
\end{equation}
with an expansion of positivity argument. These tools play a fundamental
role in the derivation of Harnack-type estimates for linear and nonlinear
parabolic equations; see, among others, \cite{DGV1,DGV2,DGV3,BDLS23,BDGLS}.  
Throughout the present paper we work in the fast diffusion regime $q>1$, in which solutions exhibit infinite speed of propagation, as in the case of the heat equation.

The local theory described above provides a natural benchmark for the fractional porous medium equation studied in this paper. In the nonlocal setting, however, the results must account for the influence of exterior values of the solution, which is commonly quantified through the so-called tail; see \cite{DKP14,DKP16}.  Our results are established under the optimal
tail assumption. More precisely, for a cylinder
$Q:=B_R(x_0)\times I$, we will consider the $L^1$-in-time tail
\begin{equation}\label{tail}
\operatorname{Tail}(u;Q)
:=
\int_I\int_{\mathbb{R}^N\setminus B_R(x_0)}
\frac{|u(x,t)|}{|x-x_0|^{N+2s}}\,\dx \, \dt.
\end{equation}
Regularity results for nonlinear nonlocal parabolic
equations, including fractional $p$-Laplacian and doubly nonlinear models,
can be found in
\cite{BGK22,StromqvistLB19,DZZ21,ByunKim24,L1,L2,L3, L4, LW25, DieningKimLeeNowak25,KM, KW, Stromqvist19} and references therein. Fractional porous medium equations, in particular, have been investigated in
\cite{FDF, BV15,BSV15,BFV18}. The present paper establishes an integral Harnack inequality and expansion of positivity  for fractional porous medium-type
equations.  The integral estimate controls the local supremum of a solution in terms of its spatial $L^q$-mass at suitable time levels, the intrinsic scale of the equation and the exterior contributions. Its proof combines a quantitative local boundedness estimate with an $L^q$-$L^q$ estimate adapted from \cite{CCI} to the nonlinear time derivative.  The first ingredient  was recently established in \cite{FDF} in the range
\begin{equation}\label{critical-range}
    1<q<q_c,
    \qquad
    q_c:=\frac{N+2s}{N-2s},
\end{equation}
which is the fractional counterpart of the local bound  \eqref{bound}. The same work also establishes qualitative local boundedness at the critical exponent $q_c$ and, under a suitable higher-integrability assumption, a quantitative estimate at and above the critical threshold. The $L^q$-$L^q$ estimate is established below in Proposition \ref{lq-lq}.

The expansion of positivity lies at the heart of any form of Harnack estimate. In the present setting, information on the measure of a positivity set at a time level $t_0$ propagates forward in time and expands from $B_\varrho(x_0)$ to $B_{2\varrho}(x_0)$ on the intrinsic time scale
\[
    k^{q-1}\varrho^{2s}.
\]
The proof is based on a Caccioppoli energy estimate and a sequence of De Giorgi-type arguments.

 We first state the integral Harnack inequality. Its proof uses the quantitative local boundedness estimate \eqref{tesi1} with integrability exponent $\texttt{r}=q$, which accounts for 
 the bound
 \begin{equation}\label{onq}
    1<q<\frac{N}{N-2s}.
\end{equation}

\begin{theorem}[Integral Harnack inequality]\label{lq-linf}
Let $q$ satisfy  \eqref{onq}, and let $u$ be a local, nonnegative weak solution to  \eqref{eq:PDE} in $\Omega_T$. Then, for every cylinder
\(
    Q_{4\varrho,\sigma}(x_0,t_0)
    :=
    B_{4\varrho}(x_0)\times(t_0-\sigma,t_0]
    \Subset\Omega_T,
\)
there exists a positive constant $\gamma$, depending only on $N,q,s,\textnormal{\texttt{L}}$, such that
\begin{align} \label{intH}
    \esssup_{Q_{\varrho/2,\sigma/2}(x_0,t_0)}u
    &\leq
    \gamma
    \left[
        \inf_{\tau\in[t_0-\sigma,t_0]}
        \int_{B_{2\varrho}(x_0)}u^q(x,\tau)\,\dx
    \right]^{\frac{2s}{\Lambda_q}}
    \sigma^{-\frac{N}{\Lambda_q}}
    \notag \\ & \qquad + \gamma \left(\frac{\sigma}{\rho^{2s}}\right)^{\frac{1}{q-1}}
\left[
1+
\left(
\left(\frac{\sigma}{\rho^{2s}}\right)^{-\frac{q}{q-1}}
\operatorname{Tail}\left(u_-;Q_{\rho/2,\sigma}(x_0,t_0)\right)
\right)^{\frac{2sq}{(q-1)\Lambda_q}}
\right] \notag \\ & \qquad + \gamma
    \left(\frac{\sigma}{\varrho^{2s}}\right)^{-\frac{N}{\Lambda_q}}
    \operatorname{Tail}^{\frac{2s}{\Lambda_q}}
    \bigl(u_+;Q_{\rho,\sigma}(x_0,t_0)\bigr)
+\gamma\,
    \operatorname{Tail}^{\frac1q}
    \bigl(u_+;Q_{\varrho/2,\sigma}(x_0,t_0)\bigr),
\end{align}
where $\Lambda_q$ is defined according to  \eqref{Lambda}.
\end{theorem}
\noindent
We next turn to the expansion of positivity. Since, in the nonlocal setting, exterior values enter the Caccioppoli estimate through a tail term, the corresponding expansion of positivity result naturally takes the form of an alternative
involving the tail: either the exterior tail is sufficiently large, or positivity propagates to a larger space-time region. Compared with the
 expansion of positivity obtained by Misawa \& Yamaura
\cite{MY}  for doubly nonlinear parabolic 
fractional equations, a relevant feature of our result is that the exterior
contribution is controlled through optimal tail assumption.

\begin{theorem}[Expansion of positivity]\label{prop:expansion_positivity}
Let $q>1$ and  $Q:= B_{R}(x_0) \times (t_1, t_2] \Subset \Omega_T$. Let $u$ be a local, nonnegative in $Q$, weak  super-solution to  \eqref{eq:PDE} in $\Omega_T$.
 Suppose that, for some $\alpha\in(0,1)$ and $k>0$, there holds
\[
    \bigl|
        \{u(\cdot,t_0)\geq k\}
        \cap B_\varrho(x_0)
    \bigr|
    \geq
    \alpha|B_\varrho|.
\]
Then, there exist constants $\Theta \equiv \Theta (N,s,q,\texttt{L}) >1$ and $\delta,\eta\in(0,1)$, depending only on $N,q,s,\textnormal{\texttt{L}}$ and on $\alpha$, such that either
$$       \Theta 
        \operatorname{Tail}^{\frac{1}{q}}\bigl(u_{-};Q\bigr)
    >
    \eta k,
    $$
or
\[
    u\geq\eta k
    \quad\mbox{a.e. in }\;
    B_{2\varrho}(x_0)\times
    \left(
        t_0+\frac12\delta k^{q-1}\varrho^{2s},
        t_0+\delta k^{q-1}\varrho^{2s}
    \right],
\]
provided
\[
    B_{8\varrho}(x_0)\times
    \left(
        t_0,
        t_0+\delta k^{q-1}(4\varrho)^{2s}
    \right]
    \subset Q.
\]
Moreover,
$ \delta\approx\alpha^{N+2s+1}$, $ \eta\approx\alpha^\beta$
for some $\beta>1$ depending only on the data.
\end{theorem}
\noindent
The paper is organized as follows. In Section~\ref{sec2} we introduce the notation, recall the notion of weak solution, and collect the preliminary results used in the sequel. Section~\ref{sec3} is devoted to the integral Harnack inequality, while in Section~\ref{sec4} we establish the expansion of positivity.  Finally, Section~\ref{sec5} develops the time-mollification arguments needed to justify the use of time-dependent test functions.

\section{Preliminaries} \label{sec2}
\noindent
This section introduces the basic notation and collects several auxiliary results that will be used throughout the paper, including the local boundedness estimate needed in the proof of the integral Harnack inequality.
\subsection*{Notation}
The symbol $\gamma$ will
denote a generic positive constant, not necessarily the same at each occurrence, that can be determined by the data from the problem under consideration. The data are represented by the parameters 
$$N,s,q,\texttt{L}.$$
We use symbols "$\approx, \lesssim, \gtrsim$" with subscripts, to indicate that a certain inequality holds up to constants whose relevant dependencies are marked in the suffix. Let $\mathbb{N}_0 = \mathbb{N} \cup \{ 0 \}.$ With $B_R(x_0)$ we denote the open ball of center $x_0 \in \R^N$ and radius $R>0$. When $x_0=0$ we simply write $B_R$. Moreover, we define a general backward parabolic cylinder as
$$
Q_{R,S}(x_0, t_0):=B_R(x_0)\times(t_0-S,t_0], \qquad (x_0,t_0) \in \R^N \times \R.
$$
For $v \in L^1(\Omega_T, \R^k)$ and a measurable subset $U\subset \Omega$ with positive $\mathcal{L}^N$-measure, we define the slice-wise mean $[v]_U :(0,T) \to \R^k$ of $v$  on $U$ as
$$ [v]_U(t):= \fint_U v(\cdot,t)\,\dx = \frac{1}{|U|} \int_{U}v(\cdot,t)\,\dx, \qquad \text{for a.e. } t \in (0,T).$$
Similarly, for a measurable set $Q \subset \Omega_T$ of positive $\mathcal{L}^{N+1}$-measure, we define the mean value $[v]_Q$ of $v$ on $Q$ as
\begin{equation} \label{vq}
     [v]_Q := \fiint_Q v\,\dx\,\dt=\frac{1}{|Q|} \iint_{Q} v\,\dx\,\dt.
\end{equation}
For $s\in(0,1)$, we introduce the fractional Sobolev space
$W^{s,2}(\R^N)$ defined by
\[
W^{s,2}(\R^N)
:= \left\{
v\in L^2(\R^N) :
\iint_{\R^N \times \R^N}
\frac{|v(x)-v(y)|^2}{|x-y|^{N+2s}}\dx\dy < \infty
\right\},
\]
which is endowed with the norm
\[
\|v\|_{W^{s,2}(\R^N)}
:= \left(\int_{\R^N} |v|^2\dx\right)^{\frac12}
+ \left(\iint_{\R^N \times \R^N}
\frac{|v(x)-v(y)|^2}{|x-y|^{N+2s}}\dx\dy\right)^{\frac12}.
\]
For a domain $\Omega \subset \R^N$, the fractional Sobolev space $W^{s,2}(\Omega)$ can be defined similarly. Moreover, we denote
\[
W^{s,2}_0(\Omega)
:=
\left\{
v \in W^{s,2}(\mathbb{R}^N)
\,:\,
v = 0 \quad \text{almost  everywhere in } \mathbb{R}^N \setminus \Omega
\right\}.
\]
We next give the notion of weak super(sub)-solution.
\begin{definition} \label{def}
A measurable function $u: \R^N \times (0,T] \to \R$ satisfying
$$ u \in C_{\loc} (0, T; L^{q+1}_{\loc}( \Omega) ) \cap L^2_{\loc} (0, T; W^{s,2}_{\loc}( \Omega) ), $$
is a local, {weak super(sub)-solution} to the equation  \eqref{eq:PDE} in $\Omega_T$, if for every sub-interval $[t_1,t_2] \subset (0,T]$, every bounded open set $\tilde{\Omega} \Subset \Omega$, it holds
\begin{align} \label{ineq:L_1_tail_condition}
    \int_{t_1}^{t_2} \int_{\R^N} \frac{|u(x,t)|}{1+|x|^{N+2s}}\,\dx\,\dt < \infty
\end{align}
and 
\begin{align}  \label{eq}
& \int_{{\tilde{\Omega}}} |u|^{q-1}u \phi\,\dx \big |_{t_1}^{t_2}-
\int_{t_1}^{t_2}\int_{ {\tilde{\Omega}}}  |u|^{q-1}u \, \partial_t\phi\,\dx\,\dt \notag \\ & \qquad \qquad + \int_{t_1}^{t_2} \iint_{\R^N \times \R^N}(u(x,t)-u(y,t))(\phi(x,t)-\phi(y,t)) K(x,y,t)\,\dx\, \dy \,\dt \geq (\leq) \, 0
\end{align}
is satisfied for all nonnegative test functions
\[
\phi \in  L^2_{\loc} (0, T; W^{s,2}_0({\tilde{\Omega}}) ) \cap  W^{1,q+1}_{\loc} (0, T; L^{q+1}({\tilde{\Omega}})) .
\]
A function that is both a weak super- and sub-solution is called a weak solution.
\end{definition}
\subsection*{Useful results}
Now, for $u,k \in \mathbb{R}$ and $q>0$ we define
\begin{align}\label{GG}
    \mathfrak{g}_{\pm}(u,k):= \pm q\int_k^u \vert \theta \vert^{q-1}(\theta-k)_{\pm}\,\d\theta,
\end{align}
where
\begin{align*}
    (\theta-k)_+:=\max\{\theta-k,0\},\qquad (\theta-k)_-:=\max\{-(\theta-k),0\}.
\end{align*}
For the following lemma we refer to~\cite[Lemma 2.2]{BDL_21}.
\begin{lemma} \label{2.3}
Let $q > 0$ and consider the function $\mathfrak{g}_{\pm}$ defined in  \eqref{GG}. Then there exists a constant $c \equiv c(q) > 0$ such that for all $a, b \in \mathbb{R}$ it holds
    \[
        \frac{1}{ c} \big( |a| + |b| \big)^{q-1} (a - b)_{\pm}^2 
        \leq \mathfrak{g}_{\pm}(a, b) 
        \leq  c \big( |a| + |b| \big)^{q-1} (a - b)_{\pm}^2.
    \]  
\end{lemma}
\noindent
Let us now state the fast geometric convergence Lemma, see \cite[Lemma 7.1]{G}.
\begin{lemma}\label{fg}
    Let $(Y_n)_{n \in \mathbb{N}_0}$ be a sequence of positive real numbers satisfying the recursive inequalities
    \[
        Y_{n+1} \leq c \, b^n \, Y_n^{1 + \alpha}
    \]
    where $c, b > 1$ and $\alpha>0$ are given numbers. 
    If 
    \[
        Y_0 \leq c^{-1/\alpha} \, b^{-1/\alpha^2},
    \]
    then $Y_n \to 0$ as $n \to \infty$.
\end{lemma}
\noindent
We next recall the parabolic Sobolev embedding lemma, see \cite[Lemma 2.5]{FDF}.
\begin{lemma}[Parabolic embedding]\label{em}
Let $s \in (0,1)$ and 
\begin{equation} \label{m}
    \mm:=\frac{2}{N}(N+(q+1)s).
\end{equation}
For any function
\[
u \in L^2\big(t_1,t_2;W^{s,2}(B_R)\big) \cap L^\infty\big(t_1,t_2;L^{q+1}(B_R)\big),
\]
which is compactly supported in $B_{(1-d)R}$ for some $d \in (0,1)$ and for almost every $t \in (t_1, t_2)$, there holds
\[
\begin{aligned}
&\int_{t_1}^{t_2} \int_{B_R} |u(x,t)|^{\mm}\,\dx\,\dt \\
&\qquad \leq c \Bigg( R^{2s} \int_{t_1}^{t_2} \iint_{B_R \times B_R}
\frac{|u(x,t)-u(y,t)|^2}{|x-y|^{N+2s}}\,\dy\,\dx\,\dt 
+ \frac{1}{d^{N+2s}} \int_{t_1}^{t_2} \int_{B_R} |u(x,t)|^2\,\dx\,\dt \Bigg) \\
&\qquad \qquad\times \Bigg( \operatorname*{ess\,sup}_{t_1 < t < t_2} \fint_{B_R} |u(x,t)|^{q+1}\,\dx \Bigg)^{\frac{2s}{N}},
\end{aligned}
\]
for a positive constant $c \equiv c(s,q,N)$.
\end{lemma}
\noindent
We will use the following two variants of the Caccioppoli estimate, which extend \cite[Proposition~2.7]{FDF}. The proof is postponed to Section \ref{sec5}.
\begin{lemma}[Caccioppoli estimates]\label{cacc}
Let $q>1$, let $u$ be a local, nonnegative, weak sub(super)-solution to \eqref{eq:PDE} in $\Omega_T$, and let $k:[t_0-S,t_0]\to(0,\infty)$ be an absolutely continuous function, nondecreasing in the subsolution case and nonincreasing in the supersolution case. Set $w_\pm(x,t):=(u(x,t)-k(t))_\pm$ and assume $Q_{R,S}(x_0,t_0)\Subset\Omega_T$. For every nonnegative, piecewise smooth cut-off function $\zeta$, compactly supported in $B_R(x_0)$ for every $t\in(t_0-S,t_0)$ and such that $\zeta(\cdot,t_0-S)=0$, we have
\begin{align}\label{cacc1}
&\esssup_{t_0-S<t<t_0}\int_{B_R(x_0)}\zeta^2(x,t)\mathfrak{g}_\pm(u(x,t),k(t))\,\dx
\notag\\
&\quad+\frac{1}{\textnormal{\texttt{L}}}\int_{t_0-S}^{t_0}\iint_{B_R(x_0)\times B_R(x_0)}
\frac{|\zeta w_\pm(x,t)- \zeta w_\pm(y,t)|^2}{|x-y|^{N+2s}}\,\dx\,\dy\,\dt
\notag\\
&\quad+\frac{1}{\textnormal{\texttt{L}}}\int_{t_0-S}^{t_0}\int_{B_R(x_0)}\zeta^2(x,t)w_\pm(x,t)
\left(\int_{\mathbb R^N}\frac{w_\mp(y,t)}{|x-y|^{N+2s}}\,\dy\right)\dx\,\dt
\notag\\
&\leq  \gamma\int_{t_0-S}^{t_0}\iint_{B_R(x_0)\times B_R(x_0)}
\max\{w_\pm^2(x,t),w_\pm^2(y,t)\}
\frac{|\zeta(x,t)-\zeta(y,t)|^2}{|x-y|^{N+2s}}\,\dx\,\dy\,\dt
\notag\\
&\quad+4\textnormal{\texttt{L}}\int_{t_0-S}^{t_0}\int_{B_R(x_0)}\zeta^2(x,t)w_\pm(x,t) \left(\int_{\mathbb R^N\setminus B_R(x_0)}\frac{w_\pm(y,t)}{|x-y|^{N+2s}}\,\dy\right)\dx\,\dt 
\notag\\
&\quad\mp \frac{1}{q} \int_{t_0-S}^{t_0}\int_{B_R(x_0)}(k^q(t))'\zeta^2(x,t)w_\pm(x,t)\,\dx\,\dt + \gamma\int_{t_0-S}^{t_0}\int_{B_R(x_0)}|\partial_t \zeta^2 (x,t)|\mathfrak{g}_\pm(u(x,t),k(t))\,\dx\,\dt,
\end{align}
where $\gamma$ depends only on the data. Moreover, for a time-independent, nonnegative, piecewise smooth cut-off function $\xi$ compactly supported in $B_R(x_0)$,  for every $\tau\in(t_0-S,t_0]$, we have
\begin{align}\label{cacc2}
&\int_{B_R(x_0)}\xi^2(x)\mathfrak{g}_\pm(u(x,\tau),k(\tau))\,\dx
+\frac{1}{\textnormal{\texttt{L}}}\int_{t_0-S}^{\tau}\iint_{B_R(x_0)\times B_R(x_0)}
\frac{|\xi(x)w_\pm(x,t)-\xi(y)w_\pm(y,t)|^2}{|x-y|^{N+2s}}\,\dx\,\dy\,\dt
\notag\\
&\quad+\frac{1}{\textnormal{\texttt{L}}}\int_{t_0-S}^{\tau}\int_{B_R(x_0)}\xi^2(x)w_\pm(x,t)
\left(\int_{\mathbb R^N}\frac{w_\mp(y,t)}{|x-y|^{N+2s}}\,\dy\right)\dx\,\dt
\notag\\
&\leq\int_{B_R(x_0)}\xi^2(x)\mathfrak{g}_\pm(u(x,t_0-S),k(t_0-S))\,\dx
\notag\\
&\quad+\gamma\int_{t_0-S}^{\tau}\iint_{B_R(x_0)\times B_R(x_0)}
\max\{w_\pm^2(x,t),w_\pm^2(y,t)\}
\frac{|\xi(x)-\xi(y)|^2}{|x-y|^{N+2s}}\,\dx\,\dy\,\dt
\notag\\
&\quad+4\textnormal{\texttt{L}}\int_{t_0-S}^{\tau}\int_{B_R(x_0)}\xi^2(x)w_\pm(x,t)
\left(\int_{\mathbb R^N\setminus B_R(x_0)}\frac{w_\pm(y,t)}{|x-y|^{N+2s}}\,\dy\right)\dx\,\dt
\notag\\
&\quad\mp \frac{1}{q} \int_{t_0-S}^{\tau}\int_{B_R(x_0)}(k^q(t))'\xi^2(x)w_\pm(x,t)\,\dx\,\dt.
\end{align}
\end{lemma}
\noindent
Now, for $\texttt{x}>0$, define
\begin{equation} \label{Lambda}
        \Lambda_\texttt{x}:=2s\texttt{x}+N(1-q).
\end{equation} We conclude with the quantitative local boundedness estimate  taken from \cite{FDF}.
 \begin{theorem} \label{th1} Assume that $u$ is a local, nonnegative, weak sub-solution to \eqref{eq:PDE} in $\Omega_T$.  If $1<q<q_c$, with $q_c$ defined in \eqref{critical-range},  then on any parabolic cylinder $ Q_{\rho,\sigma}(x_0,t_0)  \Subset \Omega_T$, it holds
\begin{align} \label{tesi1}
\operatorname*{ess\,sup}_{Q_{  \varrho/2,  \sigma/2}(x_0,t_0)} u \leq \gamma \left ( \frac{{ \varrho}^{2s}}{\sigma} \right )^{\frac{N}{\Lambda_{ {\rr}}}}\left (  \fiint_{Q_{{ \varrho},\sigma}(x_0,t_0)} u^{ {\rr}}(x,t)\,\dx\,\dt\right )^{\frac{2s}{\Lambda_{ {\rr}}}} + \gamma \left ( \frac{\sigma}{ \varrho^{2s}}\right )^{\frac{1}{q-1}} +  \gamma   \textnormal{Tail}^{\frac{1}{q}}(u; Q_{  \varrho/2, \sigma}(x_0,t_0)),
\end{align}
for every $ {\rr} \in [1, \mm]$ such that  $\Lambda_{\rr}>0$, with $\mm$ as in \eqref{m}   and $\Lambda_{\rr}$ defined according to  \eqref{Lambda},
and for a constant $\gamma \equiv \gamma (s,q,N, \textnormal{\texttt{L}})$.
If $ q \geq q_c$, assume moreover that
$$ u \in L^{ {\rr}}_{\loc}(\Omega_T), \  \text{ for some }   {\rr}> \frac{N}{2s}(q-1).$$
Then,  \eqref{tesi1} holds true
on any parabolic cylinder $ Q_{ \varrho,\sigma} (x_0,t_0)  \Subset \Omega_T$, for a constant $\gamma \equiv \gamma(s,q,N, \textnormal{\texttt{L}}, {\rr})$.
 \end{theorem}

\section{Integral Harnack Inequality} \label{sec3}
\noindent
The proof of Theorem \ref{lq-linf} combines the $L^\infty$-estimate of Theorem \ref{th1} with the $L^q$-$L^q$-estimate of Proposition  \ref{lq-lq} below. 
Its proof  relies on an anomalous energy estimate contained in Lemma \ref{la}.
We introduce the quantities that will be used in the sequel. Define,
\begin{align} \label{lambda}
    \lambda:=\frac{\Lambda_{q}}{q} = \frac{N}{q}(1-q) +2s, \quad \nu := \left( \frac{\tau - \bar{\tau}}{\rho^{2s}} \right)^{\frac{1}{q-1}},
\end{align}
and,  for \(r>0\) and \(\bar{\tau}<\tau\), we set
\begin{equation} \label{PP}
P_-(r,x_0;\bar{\tau},\tau)
:=
\max\left\{
1,
\left(\frac{\tau-\bar{\tau}}{r^{2s}}\right)^{-\frac{q}{q-1}}
\int_{\bar{\tau}}^\tau
\int_{B_{r/2}^c(x_0)}
\frac{u_-(x,t)}{|x-x_0|^{N+2s}}\,\dx\,\dt
\right\}.
\end{equation}
Whenever the spatial scale and the time interval are clear from the context, we simply write \(P_-\).
The following lemma is proved in the spirit of \cite[Lemma 4.1]{CCI}. The proof relies on a suitable test function whose admissibility is justified by the time-mollification procedure described in Section \ref{sec5}.
\begin{lemma} \label{la} Let $q>1$ and let $u$ be  a local, nonnegative, weak super-solution to \eqref{eq:PDE} in $\Omega_T$. Then for any cylinder $B_{ \rho}(x_0) \times (\bar{\tau},\tau) \Subset \Omega_{T}$ and any $\sigma \in (0,1)$, there exists a positive constant $\gamma$ depending only on the data such that
    \begin{equation}\label{tst}
\begin{aligned}
&\int_{\bar{\tau}}^\tau (t-\bar{\tau})^{\frac12}
\iint_{B_{\sigma \rho}(x_0) \times B_{\sigma \rho}(x_0)}
\frac{(u(x,t)-u(y,t))^2}{|x-y|^{N+2s}}
\bigl(\max \{u(x,t), u(y,t) \}+\nu\bigr)^{-\frac{q+1}{2}}
 \,\dx \,\dy\, \dt
\\
& \quad \leq
 \frac{\gamma \varrho^s}{(1-\sigma)^{N+2s}}
\left[
\sup_{t\in[\bar{\tau},\tau]}
\int_{B_\varrho(x_0)} u^q(x,t)\,\dx
+
\left(\frac{\tau - \bar{\tau}}{\varrho^\lambda}\right)^{\frac{q}{q-1}}
\right]^{\frac{q+1}{2q}}
\left(\frac{\tau - \bar{\tau}}{\varrho^\lambda}\right)^{\frac12}
P_-(\rho,x_0;\bar{\tau},\tau) ,
\end{aligned}
\end{equation}
    where $\lambda, \nu $ and $ P_-(\rho,x_0;\bar{\tau},\tau)$  are defined according to \eqref{lambda} and  \eqref{PP}.
\end{lemma}
\begin{proof}
    Assume for simplicity $(x_0,\bar{\tau})=(0,0)$ and set $P_{-} = P_{-}(\rho,x_0;\bar{\tau},\tau)$.  Consider $0<\sigma<\sigma'<1$, and set for brevity 
\begin{equation} \label{BB}
    B = B_\rho(0), \ \check B = B_{\sigma\rho}(0), \ \hat B = B_{\sigma'\rho}(0),
\end{equation}
then $\check B\subset\hat B\subset B$. Let $\xi\in C^1_c(\hat{B})$ be a function such that \
\begin{align}\label{property_of_tst_fct} 
\xi = 1 \ \text{in $\check B$,} \ 0 \leq \xi \leq 1, \ |\nabla\xi| \lesssim \frac{1}{(\sigma'-\sigma)\rho} \ \text{in $\R^N$}.
\end{align}
In what follows, the function
\(
(u+\nu)^{\frac{1-q}{2}}\xi^2
\)
is understood to be extended by zero outside $B$.

Now, we apply Lemma \ref{mol} and obtain the following inequality:
\begin{equation}\label{tst1}
\begin{aligned}
    0 &\leq \tau^{\frac{1}{2}} \int_{\hat{B}}\xi^2(x)  \biggl(\int_0^{u(x,\tau)} \theta^{q-1}(\theta+\nu)^{\frac{1-q}{2}} \, \mathrm{d}\theta \biggl) \,\dx - \int_0^\tau  \frac{t^{-\frac{1}{2}}}{2} \int_{\hat{B}} \xi^2 \biggl(\int_0^{u(x,t)} \theta^{q-1} (\theta+\nu)^{\frac{1-q}{2}} \, \mathrm{d}\theta \biggl) \,\dx\,\dt \\
    &\quad 
    +  \gamma \int_0^\tau t^{\frac{1}{2}} \iint_{\R^N \times \R^N} {(u(x,t)-u(y,t))}  [(u(x,t)+\nu)^{\frac{1-q}{2}}\xi^{2}(x)-(u(y,t)+\nu)^{\frac{1-q}{2}}\xi^{2}(y)]K(x,y,t) \,\dx\,\dy\,\dt \\
    &=: I_{1} + I_{2},
\end{aligned}
\end{equation}
where we denote by $I_1$ the sum of the first two terms on the right-hand side and by $I_2$ the remaining nonlocal term.
We first estimate \(I_1\) from above. Since its second term is nonpositive, we may discard it. Using the fact that $q>1$ and $\nu >0$, we obtain that
\begin{align} \label{t1}
     \int_{0}^{u} \theta^{q-1} (\theta+\nu)^{\frac{1-q}{2}} \,\mathrm{d}\theta \leq \gamma u^{\frac{q+1}{2}}.
\end{align}
Then, by  \eqref{t1}, Hölder's inequality with exponents $\left (\frac{2q}{q+1}, \frac{2q}{q-1} \right )$, the fact that $\xi \leq 1$ and \eqref{lambda}, we estimate $I_1$ as
\begin{align} \label{i1}
    I_{1} &\leq \gamma  \tau^{\frac{1}{2}} \int_{B} (u(x,\tau)+\nu)^{\frac{q+1}{2}} \xi^{2}(x) \,\dx \notag \\
    &\leq \gamma  \tau^{\frac{1}{2}} \Bigg[ \sup_{t \in [0,\tau]} \int_{B} (u(x,t)+\nu)^{q} \,\dx \Bigg]^{\frac{q+1}{2q}} \vert B \vert^{\frac{q-1}{2q}} \notag \\
    &\leq \gamma \rho^{s} \Bigl( \frac{\tau}{\rho^{\lambda}} \Bigl)^{\frac{1}{2}} \Bigg[ \sup_{t \in [0,\tau]} \int_{B} u^{q}(x,t) \,\dx + \Bigl( \frac{\tau}{\rho^{\lambda}} \Bigl)^{\frac{q}{q-1}} \Bigg]^{\frac{q+1}{2q}}.
\end{align}
Now, we proceed with the diffusion term $I_{2}$. By symmetry, we are able to rewrite this term into the following two parts
\begin{align} \label{i2}
    I_{2} &= 2 \gamma \int_0^\tau \iint_{A_{t}^{+}} {(u(x,t)-u(y,t))} t^{\frac{1}{2}} [(u(x,t)+\nu)^{\frac{1-q}{2}}\xi^{2}(x)-(u(y,t)+\nu)^{\frac{1-q}{2}}\xi^{2}(y)]K(x,y,t) \,\dx\,\dy\,\dt  \notag \\
    &\qquad 
    +2 \gamma \int_0^\tau \iint_{A_{t}^{-}} {(u(x,t)-u(y,t))} t^{\frac{1}{2}} [(u(x,t)+\nu)^{\frac{1-q}{2}}\xi^{2}(x)-(u(y,t)+\nu)^{\frac{1-q}{2}}\xi^{2}(y)]K(x,y,t) \,\dx\,\dy\,\dt \notag \\
    &=: I_{3} + I_{4},
\end{align}
where
\begin{align*}
    A_{t}^{+} &:= \big\{(x,y)\in \R^N \times \R^N :\,u(x,t)>u(y,t),\, u(y,t) \geq 0 \big\}, \\
    A_{t}^{-} &:= \big\{(x,y)\in \R^N \times \R^N :\,u(x,t)>u(y,t) ,\, u(y,t) < 0 \big\}.
\end{align*}
Therefore, for $I_3$ we obtain
\begin{align} \label{i3}
    I_{3} &\leq  +2\gamma \int_0^\tau \iint_{A_{t}^{+} \cap \{ \xi(x) \leq \xi(y) \}}
    \frac{u(x,t)-u(y,t)}{|x-y|^{N+2s}} t^{\frac{1}{2}}
    [(u(x,t)+\nu)^{\frac{1-q}{2}} -(u(y,t)+\nu)^{\frac{1-q}{2}}]
    \xi^{2}(y) \,\dx\,\dy\,\dt \notag \\
    &\qquad
    +2\gamma \int_0^\tau \iint_{A_{t}^{+} \cap \{ \xi(x) > \xi(y) \}}
    \frac{u(x,t)-u(y,t)}{|x-y|^{N+2s}} t^{\frac{1}{2}}
    [(u(x,t)+\nu)^{\frac{1-q}{2}} -(u(y,t)+\nu)^{\frac{1-q}{2}}]
    \xi^{2}(x) \,\dx\,\dy\,\dt \notag \\
    &\qquad
    + 2\gamma \int_0^\tau \iint_{A_{t}^{+} \cap \{ \xi(x) > \xi(y) \}}
    \frac{u(x,t)-u(y,t)}{|x-y|^{N+2s}} t^{\frac{1}{2}}
    (u(y,t)+\nu)^{\frac{1-q}{2}} ( \xi^{2}(x) - \xi^{2}(y))
    \,\dx\,\dy\,\dt \notag \\
    &=: I_{5}+I_{6}+I_{7}.
\end{align}
Let us estimate $I_5$-$I_7$.  By the convexity of $r \mapsto (r+\nu)^{\frac{1-q}{2}}$, we get
\begin{align*}
    (u(x,t) + \nu)^{ \frac{1-q}{2}} - (u(y,t) + \nu )^{ \frac{1-q}{2}} \leq \frac{1-q}{2} ( u(x,t) + \nu )^{- \frac{q+1}{2}} (u(x,t) - u(y,t)),
\end{align*}
then,
\begin{equation} \label{iii5}
    I_5 \leq  -\gamma \int_0^\tau
    \iint_{A_t^+\cap\{\xi(x)\leq\xi(y)\}}
    \frac{(u(x,t)-u(y,t))^2}{|x-y|^{N+2s}}
    t^{\frac12}(u(x,t)+\nu)^{-\frac{q+1}{2}}\xi^2(y)
    \,\dx\,\dy\,\dt
\end{equation}
and
\begin{align}\label{ineq:estimate_for_I_5_in_Lemma_2_1_proof}
    I_{6} \leq -\gamma  \int_0^\tau \iint_{A_{t}^{+} \cap \{\xi(x)>\xi(y) \} } \frac{(u(x,t)-u(y,t))^2}{|x-y|^{N+2s}} t^{\frac{1}{2}} (u(x,t)+\nu)^{-\frac{q+1}{2}}  \xi^{2}(x) \,\dx\,\dy\,\dt.
\end{align}
Recalling the following inequality: for all $a \geq b \geq 0$, $\epsilon \in (0,1)$ there exists $c >0$ such that
\begin{align*}
    a^{2} - b^{2} \leq \epsilon a^{2} + \frac{c}{\varepsilon} (a-b)^{2},
\end{align*}
and applying it with $a = \xi(x) > \xi(y) =b$, for any $(x,y) \in A_{t}^{+} \cap \{ \xi(x) > \xi(y) \}$, and $\epsilon \in (0,1)$ defined by
\begin{align*}
    \epsilon := \delta \frac{u(x,t) - u(y,t)}{u(y,t) + \nu} \bigg[ \frac{u(y,t) + \nu}{u(x,t) + \nu} \bigg]^{\frac{q+1}{2}} \leq \bigg[ \frac{u(y,t) + \nu}{u(x,t) + \nu} \bigg]^{\frac{q-1}{2}} < 1,
\end{align*}
with $\delta \in (0,1)$ that will be defined later,
we obtain
\begin{align*}
    \xi^{2}(x) - \xi^{2}(y) &\leq \delta \frac{u(x,t) - u(y,t)}{u(y,t) + \nu} \bigg[ \frac{u(y,t) + \nu}{u(x,t) + \nu} \bigg]^{\frac{q+1}{2}} \xi^{2}(x) \\
    &\qquad 
    + \frac{c}{\delta} \bigg[ \frac{u(y,t) + \nu}{u(x,t) - u(y,t)} \bigg] \bigg[ \frac{u(x,t) + \nu}{u(y,t) + \nu} \bigg]^{\frac{q+1}{2}} (\xi(x) - \xi(y))^{2},
\end{align*}
where $c>0$ is independent of $\delta$. While recalling that $u(y,t) + \nu \geq \nu$ on $A_t^+$, the previous inequality is used to estimate $I_{7}$ by
\begin{align}\label{ineq:estimate_for_I_6_in_Lemma_2_1_proof}
    I_{7} &\leq \gamma \delta \int_0^\tau \iint_{A_{t}^{+} \cap \{ \xi(x) > \xi(y) \}}  \frac{(u(x,t)-u(y,t))^{2}}{|x-y|^{N+2s}} t^{\frac{1}{2}} (u(x,t)+\nu)^{-\frac{q+1}{2}} \xi^{2}(x)  \,\dx\,\dy\,\dt \notag \\
    &\qquad 
    + \frac{\gamma}{\delta \nu^{q-1}} \int_0^\tau \iint_{A_{t}^{+} \cap \{ \xi(x) > \xi(y) \}} t^{\frac{1}{2}} (u(x,t)+\nu)^{\frac{q+1}{2}} \frac{(\xi(x) - \xi(y))^{2}}{|x-y|^{N+2s}} \,\dx\,\dy\,\dt.
\end{align}
Now, let $\delta \in (0,1)$ be small enough such that the first term of \eqref{ineq:estimate_for_I_6_in_Lemma_2_1_proof} can be absorbed into \eqref{ineq:estimate_for_I_5_in_Lemma_2_1_proof}. Hence, using  \eqref{iii5}, \eqref{ineq:estimate_for_I_5_in_Lemma_2_1_proof} and  \eqref{ineq:estimate_for_I_6_in_Lemma_2_1_proof}, we estimate $I_3$ in  \eqref{i3} as
\begin{align} \label{i3.2}
    I_{3} &\leq
    -\gamma \int_0^\tau
    \iint_{A_t^+\cap\{\xi(x)\leq\xi(y)\}}
    \frac{(u(x,t)-u(y,t))^2}{|x-y|^{N+2s}}
    t^{\frac12}(u(x,t)+\nu)^{-\frac{q+1}{2}}\xi^2(y)
    \,\dx\,\dy\,\dt \notag\\
    &\qquad
    - \gamma \int_0^\tau
    \iint_{A_{t}^{+} \cap \{ \xi(x) > \xi(y) \} }
    \frac{(u(x,t)-u(y,t))^{2}}{|x-y|^{N+2s}}
    t^{\frac{1}{2}} (u(x,t)+\nu)^{-\frac{q+1}{2}} \xi^{2}(x)
    \,\dx\,\dy\,\dt \notag \\
    &\qquad
    + \frac{\gamma}{\nu^{q-1}} \int_0^\tau
    \iint_{A_{t}^{+} \cap \{ \xi(x) > \xi(y) \}}
    t^{\frac{1}{2}} (u(x,t)+\nu)^{\frac{q+1}{2}}
    \frac{(\xi(x) - \xi(y))^{2}}{|x-y|^{N+2s}}
    \,\dx\,\dy\,\dt.
\end{align}
Call $I_8$ the last term in \eqref{i3.2} and observe that for all $(x,y) \in A_{t}^{+} \cap \{ \xi(x) > \xi(y) \}$ we have $\xi(x) >0$, hence $x \in \hat{B}$. Then, by splitting the space integral we deduce
\begin{align} \label{i8}
    I_{8} &\leq \frac{\gamma}{\nu^{q-1}} \int_0^\tau \iint_{\hat{B} \times B } t^{\frac{1}{2}} (u(x,t)+\nu)^{\frac{q+1}{2}} \frac{(\xi(x) - \xi(y))^{2}}{|x-y|^{N+2s}} \,\dx\,\dy\,\dt \notag \\
    &\qquad 
    + \frac{\gamma}{\nu^{q-1}} \int_0^\tau \iint_{\hat{B} \times B^c } t^{\frac{1}{2}} (u(x,t)+\nu)^{\frac{q+1}{2}} \frac{(\xi(x) - \xi(y))^{2}}{|x-y|^{N+2s}} \,\dx\,\dy\,\dt \notag \\
    &=: I_{9} + I_{10}.
\end{align}
By the gradient bound on $\xi$ in  \eqref{property_of_tst_fct}, it follows that
\begin{align*}
    I_{9} &\leq \frac{\gamma}{( \sigma^{\prime} - \sigma)^{2} \rho^{2} \nu^{q-1}} \int_0^\tau \iint_{\hat{B} \times B } t^{\frac{1}{2}} \frac{(u(x,t)+\nu)^{\frac{q+1}{2}} }{|x-y|^{N+2s-2}} \,\dx\,\dy\,\dt \\
    &\leq \frac{\gamma}{( \sigma^{\prime} - \sigma)^{2} \rho^{2} \nu^{q-1}} \Bigg[ \int_0^\tau \int_{\hat{B}} t^{\frac{1}{2}} (u(x,t)+\nu)^{\frac{q+1}{2}} \,\dx\,\dt \Bigg] \Bigg[ \sup_{x \in \hat{B}} \int_{B} \frac{\dy}{|x-y|^{N+2s-2}} \Bigg]\\
    &\leq \frac{\gamma}{( \sigma^{\prime} - \sigma)^{2} \rho^{2} \nu^{q-1}} \underbrace{\Bigg[ \int_0^\tau \int_{\hat{B}} t^{\frac{1}{2}} (u(x,t)+\nu)^{\frac{q+1}{2}} \,\dx\,\dt \Bigg]}_{=:J} \Bigg[ \int_{B_{2 \rho}  } \frac{\dz}{|z|^{N+2s-2}} \Bigg]. \nonumber
\end{align*}
At this point we focus on estimating $J$. By H\"older's inequality with exponents $\left ( \frac{2q}{q+1}, \frac{2q}{q-1}\right )$, it holds
\begin{align}\label{ineq:estimate_for_J_1_in_Lemma_2_1_proof}
    J &\leq \Bigg( \int_0^\tau \int_{\hat{B}}  (u(x,t)+\nu)^{q} \,\dx\,\dt \Bigg)^{\frac{q+1}{2q}} \Bigg( \int_0^\tau \int_{\hat{B}} t^{\frac{q}{q-1}} \,\dx\,\dt \Bigg)^{\frac{q-1}{2q}} \nonumber \\
    &\leq \Bigg( \int_0^\tau \int_{\hat{B}}  (u(x,t)+\nu)^{q} \,\dx\,\dt \Bigg)^{\frac{q+1}{2q}} \Bigg( \int_0^\tau  t^{\frac{q}{q-1}}  \,\dt \vert B \vert \Bigg)^{\frac{q-1}{2q}} \nonumber \\
    &\leq \gamma \tau^{\frac{2q-1}{2q}} \rho^{\frac{N(q-1)}{2q}} \Bigg(  \int_0^\tau \int_{\hat{B}} u^{q}(x,t) \,\dx\,\dt + \nu^{q} \tau \rho^{N} \Bigg)^{\frac{q+1}{2q}},
\end{align}
Hence, using  \eqref{ineq:estimate_for_J_1_in_Lemma_2_1_proof}, we obtain for $I_{9}$ the bound
\begin{align}\label{ineq:estimate_for_I_9_in_Lemma_2_1_proof}
    I_{9} &\leq \gamma \frac{  \tau^{\frac{2q-1}{2q}} \rho^{-s-\frac{\lambda}{2}}}{( \sigma^{\prime} - \sigma)^{2}  \nu^{q-1}} \Bigg(  \int_0^\tau \int_{\hat{B}} u^{q}(x,t) \,\dx\,\dt + \nu^{q} \tau \rho^{N} \Bigg)^{\frac{q+1}{2q}}.
\end{align}
In order to estimate $I_{10}$ in  \eqref{i8}, observe for all $x \in \hat{B}, y \in B^c$ that
\begin{align} \label{t22}
    \vert x-y \vert \geq \vert y \vert - \vert x \vert \geq (1-\sigma^{\prime} ) \vert y \vert.
\end{align}
Therefore, by \eqref{property_of_tst_fct},  \eqref{t22} and \eqref{ineq:estimate_for_J_1_in_Lemma_2_1_proof}, we obtain
\begin{align}\label{ineq:estimate_for_I_10_in_Lemma_2_1_proof}
    I_{10} &\leq \frac{\gamma}{( 1 - \sigma^{\prime})^{N+2s}  \nu^{q-1}} \int_0^\tau \iint_{\hat{B} \times B^c } t^{\frac{1}{2}} \frac{(u(x,t) + \nu )^{\frac{q+1}{2}}}{\vert y \vert^{N+2s}} \,\dx\,\dy\,\dt \nonumber \\
    &\leq \frac{\gamma}{( 1 - \sigma^{\prime})^{N+2s}  \nu^{q-1}} \Bigg[ \int_0^\tau \int_{\hat{B} } t^{\frac{1}{2}} (u(x,t)+\nu)^{\frac{q+1}{2}} \,\dx\,\dt \Bigg] \Bigg[ \int_{B^c } \frac{\dy}{\vert y \vert^{N+2s}} \Bigg] \nonumber \\
    &\leq \gamma \frac{\ \tau^{\frac{2q-1}{2q}} \rho^{\frac{N(q-1)}{2q}}}{( 1 - \sigma^{\prime})^{N+2s}  \nu^{q-1} \rho^{2s}}  \Bigg(  \int_0^\tau \int_{\hat{B}} u^{q}(x,t) \,\dx\,\dt + \nu^{q} \tau \rho^{N} \Bigg)^{\frac{q+1}{2q}} \nonumber \\
    &\leq \gamma\frac{  \tau^{\frac{2q-1}{2q}} \rho^{-s-\frac{\lambda}{2}}}{( 1 - \sigma^{\prime})^{N+2s}  \nu^{q-1}}  \Bigg(  \int_0^\tau \int_{\hat{B}} u^{q}(x,t) \,\dx\,\dt + \nu^{q} \tau \rho^{N} \Bigg)^{\frac{q+1}{2q}}.
\end{align}
Then  \eqref{ineq:estimate_for_I_9_in_Lemma_2_1_proof} and \eqref{ineq:estimate_for_I_10_in_Lemma_2_1_proof} give us for $I_{8}$ in  \eqref{i8} the bound
\begin{align}\label{ineq:estimate_for_I_8_in_Lemma_2_1_proof}
    I_{8} \leq \gamma\frac{  \tau^{\frac{2q-1}{2q}} \rho^{-s-\frac{\lambda}{2}}}{  \nu^{q-1}} \max \bigg\{ \frac{1}{( 1 - \sigma^{\prime})^{N+2s}}, \frac{1}{( \sigma^{\prime} - \sigma)^{2}} \bigg\} \Bigg(  \int_0^\tau \int_{\hat{B}} u^{q}(x,t) \,\dx\,\dt + \nu^{q} \tau \rho^{N} \Bigg)^{\frac{q+1}{2q}}.
\end{align}
Combining the estimates from above, while recalling the definition of $\nu$ in \eqref{lambda} and estimating the integral appearing in \eqref{ineq:estimate_for_I_8_in_Lemma_2_1_proof} further by taking the supremum with respect to time, yields for $I_{3}$ in  \eqref{i3.2} the estimate
\begin{align} \label{i3.3}
    I_{3} &\leq     -\gamma \int_0^\tau
    \iint_{A_t^+}
    \frac{(u(x,t)-u(y,t))^2}{|x-y|^{N+2s}}
    t^{\frac12}\Big[ u(x,t)+ \Bigl( \frac{\tau}{\rho^{2s}} \Bigl)^{\frac{1}{q-1}} \Big]^{-\frac{q+1}{2}}\max \left \{ \xi^2(x),\xi^2(y) \right \}
    \,\dx\,\dy\,\dt \notag\\ & \qquad 
    + \gamma \tau^{\frac{1}{2}} \rho^{s - \frac{\lambda}{2}} \max \bigg\{ \frac{1}{( 1 - \sigma^{\prime})^{N+2s}}, \frac{1}{( \sigma^{\prime} - \sigma)^{2}} \bigg\} \Bigg(  \sup_{t \in [0,\tau]} \int_{\hat{B}} u^{q}(x,t) \,\dx  + \Bigl( \frac{\tau}{\rho^{2s}} \Bigl)^{\frac{q}{q-1}}  \rho^{N} \Bigg)^{\frac{q+1}{2q}} 
    .
\end{align}
This completes the estimate of $I_3$ in  \eqref{i2}. We next turn to $I_4$. 
For every $(x,y)\in A_t^{-}$, we have $u(y,t)<0$. Since $u\geq 0$ in $B$, 
it necessarily follows that $y\in B^c$, and hence $\xi(y)=0$. Therefore, we estimate $I_4$ as follows:
\begin{align}\label{ineq:estimate_for_I_4_in_Lemma_2_1_proof}
    I_{4} &\leq \gamma\int_0^\tau \iint_{A_{t}^{-}} \frac{u(x,t) + u_{-}(y,t) }{\vert x-y \vert^{N+2s}} t^{\frac{1}{2}} ( u(x,t) + \nu )^{\frac{1-q}{2}} \xi^{2}(x) \,\dx\,\dy\,\dt \nonumber \\
    &\leq \gamma \frac{1}{\nu^{q-1}} \int_0^\tau \iint_{\hat{B} \times B^c} t^{\frac{1}{2}} \frac{( u(x,t) + \nu )^{\frac{q+1}{2}}}{\vert x-y \vert^{N+2s}} \,\dx\,\dy\,\dt \nonumber \\
    &\qquad 
    + \frac{\gamma}{\nu^{q}} \int_0^\tau \iint_{\hat{B} \times B^c} t^{\frac{1}{2}} \frac{( u(x,t) + \nu )^{\frac{q+1}{2}}}{\vert x-y \vert^{N+2s}} u_{-}(y,t) \,\dx\,\dy\,\dt \nonumber \\
    &=: I_{11} + I_{12}.
\end{align}
The term $I_{11}$ can be estimated similarly to  \eqref{ineq:estimate_for_I_10_in_Lemma_2_1_proof}, therefore
\begin{align*}
    I_{11} \leq \gamma \frac{ \tau^{\frac{1}{2}} \rho^{s-\frac{\lambda}{2}}}{( 1 - \sigma^{\prime})^{N+2s}  }  \Bigg(  \sup_{t \in [0,\tau]} \int_{\hat{B}} u^{q}(x,t) \,\dx  + \nu^{q}  \rho^{N} \Bigg)^{\frac{q+1}{2q}}.
\end{align*}
For $I_{12}$ we rely on a tail estimate for $u_{-}$. By  \eqref{t22} and H\"older with exponents $\left ( \frac{2q}{q+1}, \frac{2q}{q-1}\right )$, we obtain 
\begin{align*}
    I_{12} &\leq \frac{\gamma}{(1-\sigma^{\prime})^{N+2s} \nu^{q}} \int_0^\tau t^{\frac{1}{2}} \Bigg[ \int_{\hat{B}} ( u(x,t) + \nu )^{\frac{q+1}{2}} \,\dx  \Bigg] \Bigg[ \int_{B^c} \frac{u_{-}(y,t)}{\vert y  \vert^{N+2s}} \,\dy   \Bigg] \, \dt\\
    &\leq \gamma\frac{ \tau^{\frac{1}{2}}}{(1-\sigma^{\prime})^{N+2s} \nu^{q}} \int_0^\tau  \Bigg[\int_{B_{\rho/2}^c} \frac{u_-(y,t)}{|y|^{N+2s}}\,\dy\Bigg] \Bigg[ \int_{\hat{B}} ( u(x,t) + \nu )^{\frac{q+1}{2}} \,\dx  \Bigg] \,\dt \\
    &\leq \gamma\frac{\tau^{\frac{1}{2}}}{(1-\sigma^{\prime})^{N+2s} \nu^{q}} \textnormal{Tail}\bigl(u_-;B_{\rho/2}\times(0,\tau]\bigr) \Bigg[ \sup_{t \in [0,\tau]} \int_{\hat{B}} ( u(x,t) + \nu )^{\frac{q+1}{2}} \,\dx  \Bigg] \\
    &\leq \gamma\frac{ \tau^{\frac{1}{2}}}{(1-\sigma^{\prime})^{N+2s} \nu^{q}} \textnormal{Tail}\bigl(u_-;B_{\rho/2}\times(0,\tau]\bigr) \Bigg[ \sup_{t \in [0,\tau]} \int_{\hat{B}} ( u(x,t) + \nu )^{q} \,\dx  \Bigg]^{\frac{q+1}{2q}} \vert \hat{B} \vert^{\frac{q-1}{2q}} \\
    &\leq \gamma\frac{ \tau^{\frac{1}{2}} \rho^{\frac{N(q-1)}{2q}}}{(1-\sigma^{\prime})^{N+2s} \nu^{q}} \Bigg[ \sup_{t \in [0,\tau]} \int_{\hat{B}}  u^{q}(x,t) \,\dx  + \nu^{q} \rho^{N}  \Bigg]^{\frac{q+1}{2q}}   \textnormal{Tail}\bigl(u_-;B_{\rho/2}\times(0,\tau]\bigr) \\
    &\leq \gamma\frac{ \tau^{\frac{1}{2}} \rho^{s - \frac{\lambda}{2}}}{(1-\sigma^{\prime})^{N+2s} } \Bigg[ \sup_{t \in [0,\tau]} \int_{\hat{B}}  u^{q}(x,t) \,\dx  + \nu^{q} \rho^{N}  \Bigg]^{\frac{q+1}{2q}}   \Bigg[ \frac{1}{\nu^{q}} \textnormal{Tail}\bigl(u_-;B_{\rho/2}\times(0,\tau]\bigr) \Bigg].
\end{align*}
Plugging the estimates of $I_{11}$ and $I_{12}$ into  \eqref{ineq:estimate_for_I_4_in_Lemma_2_1_proof}, recalling the definitions of $P_{-}$ and $\nu$ in  \eqref{PP}, \eqref{lambda}, yields
\begin{align} \label{i4}
    I_{4} &\leq \gamma\frac{ \tau^{\frac{1}{2}} \rho^{s - \frac{\lambda}{2}}}{(1-\sigma^{\prime})^{N+2s} } \Bigg[ \sup_{t \in [0,\tau]} \int_{\hat{B}}  u^{q}(x,t) \,\dx  + \nu^{q} \rho^{N}  \Bigg]^{\frac{q+1}{2q}} P_{-} \notag \\
    &\leq \gamma\frac{  \rho^{s} }{(1-\sigma^{\prime})^{N+2s} }  \Bigg[ \sup_{t \in [0,\tau]} \int_{\hat{B}}  u^{q}(x,t) \,\dx  + \nu^{q} \rho^{N}  \Bigg]^{\frac{q+1}{2q}} \Bigl( \frac{\tau}{\rho^{\lambda}} \Bigl)^{\frac{1}{2}} P_{-}.
\end{align}
Therefore, by  \eqref{i3.3} and  \eqref{i4}, we estimate $I_2$ in  \eqref{i2} as
\begin{align} \label{i212}
    I_{2} &= I_{3} + I_{4} \notag  \\
    &\leq - \gamma \int_0^\tau \iint_{A_{t}^{+} } \frac{(u(x,t)-u(y,t))^{2}}{|x-y|^{N+2s}} t^{\frac{1}{2}} \Big[ u(x,t)+ \Bigl( \frac{\tau}{\rho^{2s}} \Bigl)^{\frac{1}{q-1}} \Big]^{-\frac{q+1}{2}} \max\{\xi^{2}(x), \xi^2(y) \}  \,\dx\,\dy\,\dt \notag \\
    &\qquad 
    +  \gamma  \rho^{s}  \max \bigg\{ \frac{1}{( 1 - \sigma^{\prime})^{N+2s}}, \frac{1}{( \sigma^{\prime} - \sigma)^{2}} \bigg\} \Bigg[ \sup_{t \in [0,\tau]} \int_{\hat{B}}  u^{q}(x,t) \,\dx  + \Bigl( \frac{\tau}{\rho^{\lambda}} \Bigl)^{\frac{q}{q-1}}  \Bigg]^{\frac{q+1}{2q}} \Bigl( \frac{\tau}{\rho^{\lambda}} \Bigl)^{\frac{1}{2}} P_{-}.
\end{align}
Finally, by combining the estimates \eqref{i1} of $I_1$ and \eqref{i212} of $I_2$ defined in \eqref{tst1}, we obtain 
\begin{equation} \label{fest}
\begin{aligned}
&\int_0^\tau t^{\frac12}
\iint_{A_t^+}
\frac{(u(x,t)-u(y,t))^2}{|x-y|^{N+2s}}
\Big[ u(x,t)+ \Bigl( \frac{\tau}{\rho^{2s}} \Bigl)^{\frac{1}{q-1}} \Big]^{-\frac{q+1}{2}}
\max\bigl\{\xi^2(x),\xi^2(y)\bigr\}
 \,\dx \,\dy\, \dt
\\
& \quad \leq
\gamma \varrho^s
\max\left\{
\frac{1}{(\sigma'-\sigma)^2},
\frac{1}{(1-\sigma')^{N+2s}}
\right\}
\left[
\sup_{t\in[0,\tau]}
\int_{B_\varrho} u^q(x,t)\,\dx
+
\left(\frac{\tau}{\varrho^\lambda}\right)^{\frac{q}{q-1}}
\right]^{\frac{q+1}{2q}}
\left(\frac{\tau}{\varrho^\lambda}\right)^{\frac12}
P_- ,
\end{aligned}
\end{equation} for a positive constant $\gamma$ depending only on the data. Now, we observe that swapping the variables $x$ and $y$ we obtain on the left-hand side of \eqref{fest}: the domain $\{(x,y) \in \R^N \times \R^N: u(y,t)>u(x,t) \geq 0 \}$ in place of $A_t^{+}$ and the term $(u(y,t)+\nu)^{-\frac{q+1}{2}}$ in place of $(u(x,t)+\nu)^{-\frac{q+1}{2}}$.

Therefore, choosing $\sigma':=(1+\sigma)/2$, summing the two inequalities and using the properties of $\xi$ we get \eqref{tst}.
\end{proof}
\noindent
Now, we use Lemma \ref{la} to prove the following $L^q$-$L^q$ estimate. 
\begin{proposition}
\label{lq-lq}
    Let $q>1$ and let $u$ be  a local, nonnegative, weak solution to \eqref{eq:PDE} in $\Omega_T$. Then for any cylinder $B_{4 \rho}(x_0) \times (\bar{\tau},\tau) \Subset \Omega_{T}$, there exists a positive constant $\gamma$ depending only on the data such that
    \begin{align} \label{prop3.2}
        \sup_{t \in [\bar{\tau},\tau]} \int_{B_{\rho}(x_0)} u^{q}(x,t) \,\dx & \leq \gamma \inf_{t \in [\bar{\tau},\tau]} \int_{B_{2 \rho}(x_0)} u^{q}(x,t) \,\dx + \gamma \Bigl( \frac{\tau-\bar{\tau}}{\rho^{\lambda}} \Bigl)^{\frac{q}{q-1}} \big( P_{-} + P_{-}^{\frac{q}{q-1}} \big) \notag \\ & \qquad + \gamma \rho^N\operatorname{Tail}(u_{+}, B_{\rho}(x_0) \times [\bar{\tau},\tau]),
    \end{align}
    where $P_-(\rho;x_0,\bar{\tau},\tau)$ and $\lambda, \nu$ are defined according to \eqref{PP}, \eqref{lambda} and $P_{-} = P_-(\rho,x_0;\bar{\tau},\tau)$.
\end{proposition}
\begin{proof}
Assume for simplicity $(x_0,\bar{\tau})=(0,0)$ and set $P_{-} = P_-(\rho,x_0;\bar{\tau},\tau)$.  For $n \in \mathbb{N}_{0}$, we introduce
\begin{align*}
    \rho_{n} = \rho \sum_{i=0}^{n} \frac{1}{2^{i}}, \quad
    \check{\rho}_{n} = \frac{3 \rho_{n} + \rho_{n+1}}{4}, \quad
    \hat{\rho}_{n} = \frac{\rho_{n} + 3 \rho_{n+1}}{4},
\end{align*}
and the balls
\begin{align*}
    B_{n} = B_{\rho_{n}} (0), \quad \check{B}_n = B_{\check{\rho}_{n}} (0), \quad \hat{B}_{n} = B_{\hat{\rho}_{n}} (0).
\end{align*}
Then, $B_{n} \subset \check{B}_n \subset \hat{B}_{n} \subset B_{n + 1} \subset B_{2 \rho}$. Let $\xi \in C_{c}^{\infty}(\check{B}_{n})$ be such that
\begin{align} \label{xi}
    \xi = 1 \text{ in } {B}_{n},\quad 0 \leq \xi \leq 1,\quad \vert \nabla \xi \vert \lesssim \frac{ 2^{n}}{\rho} \text{ in } \mathbb{R}^{N}.
\end{align}
Choosing as test function $\phi:=\xi$ in \eqref{eq}, for $t_1 \leq t_2$ we obtain
\begin{align} \label{3.24}
    0
    &= \int_{\check{B}_{n}} \vert u \vert^{q-1} u
    \xi \,\dx \bigg\vert_{t_1}^{t_2} 
    + \int_{t_1}^{t_2}
    \iint_{\mathbb{R}^{N} \times \mathbb{R}^{N}}
    (u(x,t)-u(y,t))
    (\xi(x)-\xi(y))
    K(x,y,t)
    \,\dx\,\dy\,\dt .
\end{align}
If instead $t_1>t_2$, we apply the weak formulation in the time interval
$(t_2,t_1)$ and multiply the resulting identity by $-1$. Therefore,
\eqref{3.24} continues to hold also in this case. Now, set
\begin{align*}
      S_n:= \sup_{ t \in [0 ,\tau]} \int_{B_{n}} u^{q}(x,t) \,\dx 
\end{align*}
and choose $t_2$ to satisfy
\begin{align*}
    \int_{B_{2\rho}} u^{q}(x,t_2) \,\dx = \inf_{t \in [0,\tau]} \int_{B_{2\rho}} u^{q}(x,t) \,\dx =: I.
\end{align*}
Therefore, by the arbitrariness of $t_1$ in  \eqref{3.24}, we get from  \eqref{3.24}, 
\begin{align}\label{ineq:prelim_estimate_of_S_n_+_1}
    S_{n} \leq I + \int_{t_1}^{t_2} \iint_{\mathbb{R}^{N} \times \mathbb{R}^{N}}   (u(x,t) - u(y,t)) ( \xi (x) - \xi (y) )K(x,y,t)  \,\dx\,\dy\,\dt.
\end{align}
Denote by $J_1$ the diffusion term on the right-hand side.  Assume first that $t_1\leq t_2$. Using \eqref{xi}, \eqref{k} and the symmetry of the integrand, we split $J_1$ as follows
\begin{align} \label{jj1}
    J_{1} &\leq \gamma \int_{t_1}^{t_2} \iint_{\hat{B}_{n} \times \hat{B}_{n}}   (u(x,t) - u(y,t)) ( \xi (x) - \xi (y) )K(x,y,t)  \,\dx\,\dy\,\dt \notag \\
    &\qquad 
    + 2 \gamma \int_{t_1}^{t_2} \iint_{\hat{B}_{n} \times \hat{B}_{n}^{c}}  (u(x,t) - u(y,t))  \xi (x)K(x,y,t)   \,\dx\,\dy\,\dt \notag \\
    &\leq \gamma \int_{t_1}^{t_2} \iint_{\hat{B}_{n} \times \hat{B}_{n}} ( u(x,t) - u(y,t)) ( \xi (x) - \xi (y) ) K(x,y,t) \,\dx\,\dy\,\dt \notag \\
    &\qquad 
    + \gamma \int_{t_1}^{t_2} \iint_{\hat{B}_{n} \times \hat{B}_{n}^{c}} \frac{u(x,t)}{\vert x-y \vert^{N+2s}} \xi(x) \,\dx\,\dy\,\dt + \gamma \int_{t_1}^{t_2} \iint_{\hat{B}_{n} \times \hat{B}_{n}^{c}} \frac{u_{-}(y,t)}{\vert x-y \vert^{N+2s}} \xi(x) \,\dx\,\dy\,\dt \notag \\
    &=: J_{2} + J_{3} + J_{4},
\end{align}
for a constant $\gamma$ depending on data. We estimate $J_{2}$-$J_{4}$ separately. We analyze first $J_{2}$. Note that $u(\cdot,t) \geq 0$ in $\hat{B}_n$, for any $t \in [0,\tau]$. Additionally, the integrand in $J_{2}$ is positive, if and only if the differences $u(x,t) - u(y,t), \xi(x) - \xi(y)$ have the same sign. Therefore, by \eqref{k}
\begin{align}\label{ineq:prelim_estimate_of_J_3}
    J_{2} \leq \gamma \int_{0}^{\tau} \iint_{A_1} \frac{ u(x,t) - u(y,t)  }{\vert x-y \vert^{N+2s}} ( \xi (x) - \xi (y) )  \,\dx\,\dy\,\dt,
\end{align}
where
\begin{align*}
    A_1 = \{ (x,y) \in \hat{B}_{n} \times \hat{B}_{n}: u(x,t) > u(y,t), \xi(x) > \xi(y) \}, \quad \forall t \in [0,\tau].
\end{align*}
Moreover, for all $(x,y) \in A_1$, using \eqref{xi}, we have
\begin{align}\label{ineq:estimate_cut_0ff-function_1}
    0 \leq \xi(x) - \xi(y) \leq \gamma\frac{2^{n}  }{\rho} \vert x-y \vert,
\end{align}
with $\gamma$ independent of $n$. Let $\nu := \Bigl( \frac{\tau}{\rho_{n+1}^{2s}} \Bigl)^{\frac{1}{q-1}}$. Combining \eqref{ineq:estimate_cut_0ff-function_1} with \eqref{ineq:prelim_estimate_of_J_3} and applying weighted Hölder's inequality with exponents $\left (2,2 \right )$  yields 
\begin{align}\label{ineq:second_prelim_estimate_of_J_3}
    J_{2} &\leq \gamma\frac{ 2^{n}}{\rho} \int_{0}^{\tau} \iint_{A_1} \frac{ u(x,t) - u(y,t)  }{\vert x-y \vert^{N+2s-1}}    \,\dx\,\dy\,\dt \nonumber \\
    &\leq \gamma\frac{ 2^{n}}{\rho} \int_{0}^{\tau} \iint_{A_1} ( u(x,t) - u(y,t))(u(x,t) + \nu)^{- \frac{q+1}{4}} t^{\frac{1}{4}}  \nonumber \\
    &\phantom{=========}
    \cdot (u(x,t) + \nu)^{ \frac{q+1}{4}} t^{-\frac{1}{4}} \vert x-y \vert \frac{1 }{\vert x-y \vert^{N+2s}} \,\dx\,\dy\,\dt \nonumber \\
    &\leq \gamma\frac{2^{n}}{\rho} \Bigg[ \int_{0}^{\tau} t^{\frac{1}{2}} \iint_{A_1} \frac{( u(x,t) - u(y,t) )^{2} }{\vert x-y \vert^{N+2s}} [u(x,t) +  \nu ]^{- \frac{q+1}{2}}   \,\dx\,\dy\,\dt \bigg]^{\frac{1}{2}} \nonumber \\
    &\phantom{====}
    \cdot \Bigg[ \int_{0}^{\tau} t^{-\frac{1}{2}} \iint_{A_1} \frac{(u(x,t) + \nu)^{\frac{q+1}{2}}}{\vert x-y \vert^{N+2s-2}}  \,\dx\,\dy\,\dt \Bigg]^{\frac{1}{2}} =: \gamma\frac{ 2^{n}}{\rho} J_{5}^{\frac{1}{2}} J_{6}^{\frac{1}{2}}.
\end{align}
In order to estimate $J_{5}$ further we apply \eqref{tst} with 
\begin{equation} \label{isigma}
    (\rho, \sigma)=\left (\rho_{n+1}, \frac{{\hat{\rho}}_{n}}{\rho_{n+1}} \right ).
\end{equation}
This yields $\sigma \in (0,1)$ and 
\begin{align*}
    1- \sigma \gtrsim \frac{ \rho_{n+1} -  \hat{\rho}_{n}}{ \rho_{n+1}} \gtrsim \frac{1}{2^{n}}.
\end{align*}
Therefore, using also the fact that $\rho \leq \rho_{n+1} \leq 2\rho$, we obtain
\begin{align}\label{ineq:estimate_of_J_6}
    J_{5} &\leq \gamma \rho_{n+1}^{s} 2^{n(N+2s)}   \left[ \sup_{t \in [0,\tau]} \int_{B_{\rho_{n+1}}} u^{q}(x,t) \,\dx + \left( \frac{\tau}{\rho_{n+1}^{\lambda}} \right)^{\frac{q}{q-1}} \right]^{\frac{q+1}{2q}} \left( \frac{\tau}{\rho_{n+1}^{\lambda}} \right)^{\frac{1}{2}} P_{-} \nonumber \\
    &\leq \gamma b^{n} \rho^{s} \left[ S_{n+1} + \left( \frac{\tau}{\rho^{\lambda}} \right)^{\frac{q}{q-1}} \right]^{\frac{q+1}{2q}} \left( \frac{\tau}{\rho^{\lambda}} \right)^{\frac{1}{2}} P_{-},
\end{align}
where $P_{-}$ is defined according to \eqref{PP}, and $b>1$ depending only on the data. Next, we estimate $J_{6}$ by
\begin{align}\label{ineq:estimate_of_J_7}
    J_{6} &\leq \Bigg[ \int_{0}^{\tau} t^{-\frac{1}{2}} \,\dt \Bigg] \Bigg[ \sup_{t \in [0,\tau]} \int_{B_{n+1}} ( u(x,t) + \nu )^{\frac{q+1}{2}} \,\dx \Bigg] \Bigg[ \sup_{x \in B_{n+1}} \int_{B_{n+1}} \frac{1}{\vert x-y \vert^{N+2s-2}} \,\dy \Bigg] \nonumber\\
    &\leq \gamma \tau^{\frac{1}{2}} \Bigg[ \sup_{t \in [0,\tau]} \int_{B_{n+1}} ( u^{q}(x,t) + \nu^{q} ) \,\dx \Bigg]^{\frac{q+1}{2q}} \vert B_{n+1} \vert^{\frac{q-1}{2q}} \Bigg[ \int_{B_{4 \rho}  } \frac{1}{\vert z \vert^{N+2s-2}} \,\dz \Bigg] \nonumber\\
    &\leq \gamma \tau^{\frac{1}{2}} \rho^{\frac{N(q-1)}{2q}+2-2s} \Bigg[ \sup_{t \in [0,\tau]} \int_{B_{n+1}}  u^{q}(x,t) \,\dx + \nu^{q} \rho_{n+1}^{N} \Bigg]^{\frac{q+1}{2q}} \nonumber\\
    &\leq \gamma \rho^{2-s} \Bigg[ S_{n+1} + \Bigl( \frac{\tau}{\rho^{\lambda}} \Bigl)^{\frac{q}{q-1}} \Bigg]^{\frac{q+1}{2q}} \Bigl( \frac{\tau}{\rho^{\lambda}} \Bigl)^{\frac{1}{2}},
\end{align}
where we used H\"older's inequality with exponents $\left (\frac{2q}{q+1}, \frac{2q}{q-1} \right )$.
Inserting the estimates \eqref{ineq:estimate_of_J_6} and \eqref{ineq:estimate_of_J_7} into \eqref{ineq:second_prelim_estimate_of_J_3}, we obtain that there exist $\gamma, b >1$, depending only on the data, such that
\begin{align*}
    J_{2} \leq \gamma b^{n} \Bigg[ S_{n+1} + \Bigl( \frac{\tau}{\rho^{\lambda}} \Bigl)^{\frac{q}{q-1}} \Bigg]^{\frac{q+1}{2q}} \Bigl( \frac{\tau}{\rho^{\lambda}} \Bigl)^{\frac{1}{2}} P_{-}^{\frac{1}{2}}.
\end{align*}
Now, applying the scaled Young's inequality with $(\frac{2q}{q+1}, \frac{2q}{q-1})$ and $\epsilon \in (0,1)$, which will be chosen later, yields
\begin{align}\label{ineq:estimate_of_J_3}
    J_{2} \leq \epsilon \Bigg[ S_{n+1} + \Bigl( \frac{\tau}{\rho^{\lambda}} \Bigl)^{\frac{q}{q-1}} \Bigg] + \gamma_{\epsilon} b^{n} \Bigl( \frac{\tau}{\rho^{\lambda}} \Bigl)^{\frac{q}{q-1}} P_{-}^{\frac{q}{q-1}},
\end{align}
where $\gamma_{\epsilon} \to \infty$ as $\epsilon \to 0$ and $b>1$ depending only on the data. For the estimate of $J_{3}$ in  \eqref{jj1}, using \eqref{xi}, we easily observe that
\begin{align*}
    J_{3} \leq \gamma \int_{0}^{\tau} \iint_{\check{B}_{n} \times \hat{B}_{n}^{c}} \frac{u(x,t)}{\vert x-y \vert^{N+2s}} \,\dx\,\dy\,\dt.
\end{align*}
Now, note that for any $x \in \check{B}_{n}, y \in \hat{B}_{n}^{c}$ there holds
\begin{align} \label{th}
    \vert x-y \vert \geq \vert y \vert - \vert x \vert \gtrsim \frac{\vert y \vert}{2^{n}}.
\end{align}
Therefore, by  \eqref{th}, Hölder's inequality and scaled Young's inequality with exponents $\left ( q, \frac{q}{q-1}\right )$, $\rho_{n+1} \approx \rho$, we obtain 
\begin{align}\label{ineq:estimate_of_J_4}
    J_{3} &\leq \gamma \tau \Bigg[ \sup_{t \in [0,\tau]} \int_{\check{B}_{n}} u(x,t) \,\dx \Bigg] \Bigg[ \sup_{x \in \check{B}_{n}} \int_{\hat{B}_{n}^{c}} \frac{1}{\vert x-y \vert^{N+2s}} \,\dy \Bigg] \nonumber \\
    &\leq \gamma 2^{(N+2s)n} \tau \Bigg[ \sup_{t \in [0,\tau]} \int_{B_{n+1}} u(x,t) \,\dx \Bigg] \Bigg[  \int_{B_{\rho}^{c}} \frac{1}{\vert z \vert^{N+2s}} \,\dz \Bigg] \nonumber \\
    &\leq \gamma\frac{b^{n} \tau}{\rho^{2s}} \Bigg[ \sup_{t \in [0,\tau]} \int_{B_{n+1}} u^{q}(x,t) \,\dx \Bigg]^{\frac{1}{q}} \vert B_{n+1} \vert^{\frac{q-1}{q}} \nonumber \\
    &\leq \gamma\frac{ b^{n} \tau}{\rho^{\lambda}} S_{n+1}^{\frac{1}{q}} \leq \epsilon S_{n+1} + \gamma_{\epsilon} b^{n} \Bigl( \frac{\tau}{\rho^{\lambda}} \Bigl)^{\frac{q}{q-1}},
\end{align}
where $\epsilon>0$ will be determined later.
Here, we also have that $\gamma_{\epsilon} \to \infty$ as $\epsilon \to 0$ and $b>1$ depending only on the data. The term $J_{4}$ can be estimated similarly, recalling the definition in \eqref{PP}, using \eqref{xi}, \eqref{th} and $\check{\rho}_n \approx \rho$, we get
\begin{align}\label{ineq:estimate_of_J_5}
    J_{4} &\leq \gamma \int_{0}^{\tau} \iint_{\check{B}_{n} \times \hat{B}_{n}^{c}} \frac{u_{-}(y,t)}{\vert x-y \vert^{N+2s}}  \,\dx\,\dy\,\dt \nonumber \\
    &\leq \gamma 2^{(N+2s)n} \vert \check{B}_{n} \vert \int_{0}^{\tau} \int_{\hat{B}_{n}^{c}} \frac{u_{-}(y,t)}{\vert y \vert^{N+2s}}  \,\dy\,\dt \nonumber\\
    &\leq \gamma b^{n} \rho^{N} \left ( \frac{\tau}{\varrho^{2s}}\right )^{\frac{q}{q-1}} \left ( \frac{\tau}{\varrho^{2s}}\right )^{-\frac{q}{q-1}} \int_{0}^{\tau} \int_{B_{\frac{\rho}{2}}^{c}}  \frac{u_{-}(y,t)}{\vert y  \vert^{N+2s}}  \,\dy\,\dt \nonumber\\
    &\leq \gamma b^{n} \Bigl( \frac{\tau}{\rho^{\lambda}} \Bigl)^{\frac{q}{q-1}} \Bigl( \frac{\tau}{\rho^{2s}} \Bigl)^{-\frac{q}{q-1}} \int_{0}^{\tau} \int_{B_{\frac{\rho}{2}}^{c}}  \frac{u_{-}(y,t)}{\vert y \vert^{N+2s}}  \,\dy\,\dt \nonumber\\
    &\leq \gamma b^{n} \Bigl( \frac{\tau}{\rho^{\lambda}} \Bigl)^{\frac{q}{q-1}} P_{-},
\end{align}
with $b>1$ depending only on the data. Therefore, in the case  $t_1\leq t_2$, combining \eqref{ineq:estimate_of_J_3}, \eqref{ineq:estimate_of_J_4} and \eqref{ineq:estimate_of_J_5}, we get for $J_1$ in \eqref{ineq:prelim_estimate_of_S_n_+_1},
\begin{align}\label{forwardJ1}
    J_{1}
    \leq
    \epsilon S_{n+1}
    +
    \gamma_\epsilon b^n
    \Bigl(\frac{\tau}{\rho^\lambda}\Bigr)^{\frac q{q-1}}
    \left(
    P_-+P_-^{\frac q{q-1}}
    \right),
\end{align}
for some $\epsilon\in(0,1)$ to be chosen later, with
$\gamma_\epsilon\to\infty$ as $\epsilon\to0$ and $b>1$ depending
only on the data. Assume now that $t_1>t_2$. 
In this case, we denote by $J_1'$ the diffusion term on the right hand side of
\eqref{ineq:prelim_estimate_of_S_n_+_1}, and observe that 
\begin{align*}
    J_1'
    &=
    -\int_{t_2}^{t_1}
    \iint_{\mathbb R^N\times\mathbb R^N}
    (u(x,t)-u(y,t))
    (\xi(x)-\xi(y))
    K(x,y,t)
    \,\dx\,\dy\,\dt.
\end{align*}
Using again the properties of $\xi$ and the symmetry of the
integrand, we obtain
\begin{align*}
    J_1'
    &\leq
    \gamma\int_{t_2}^{t_1}
    \iint_{\hat B_n\times\hat B_n}
    (u(x,t)-u(y,t))
    (\xi(y)-\xi(x))
    K(x,y,t)
    \,\dx\,\dy\,\dt\\
    &\qquad
    +2\gamma\int_{t_2}^{t_1}
    \iint_{\hat B_n\times\hat B_n^c}
    (u(y,t)-u(x,t))
    \xi(x)K(x,y,t)
    \,\dx\,\dy\,\dt =: J_2' + J_3'.
\end{align*}
The first integrand is positive only when the differences
$u(x,t)-u(y,t)$ and $\xi(x)-\xi(y)$ have opposite signs. By
symmetry, we may restrict to
$$
A_2:= \{(x,y) \in \hat B_n\times\hat B_n: u(x,t)>u(y,t), \xi(x)<\xi(y) \}.
$$
Therefore, arguing as in
\eqref{ineq:second_prelim_estimate_of_J_3}, we obtain
\begin{align*}
J_2' &\leq
\gamma\frac{2^n}{\rho}
\Bigg[
\int_0^\tau t^{\frac12}
\iint_{A_2}
\frac{(u(x,t)-u(y,t))^2}
{|x-y|^{N+2s}}
[u(x,t)+\nu]^{-\frac{q+1}{2}}
\,\dx\,\dy\,\dt
\Bigg]^{\frac12}\\
&\qquad\quad\cdot
\Bigg[
\int_0^\tau t^{-\frac12}
\iint_{A_2}
\frac{(u(x,t)+\nu)^{\frac{q+1}{2}}}
{|x-y|^{N+2s-2}}
\,\dx\,\dy\,\dt
\Bigg]^{\frac12}  =: \gamma\frac{ 2^{n}}{\rho} (J_{5}')^{\frac{1}{2}} (J_{6}')^{\frac{1}{2}}.
\end{align*}
Now, $J_5'$ is controlled using \eqref{tst} with the choice \eqref{isigma},  while $J_6'$ is estimated
exactly as $J_6$. Therefore, the interior contribution $J_2'$ satisfies
the same estimate as $J_2$ in \eqref{ineq:estimate_of_J_3}. We next estimate the exterior contribution $J_3'$. Since
$u(x,t)\geq0$ for $x\in\hat B_n$, we have
\[
u(y,t)-u(x,t)\leq u_+(y,t).
\]
Thus, using \eqref{th}, $\check{\rho}_n \approx \rho$ and the fact that
$\operatorname{supp}\xi\subset\check B_n$, we obtain
\begin{align*}
J_3' & \leq 
\gamma\int_0^\tau
\iint_{\check B_n\times\hat B_n^c}
\frac{u_+(y,t)}
{|x-y|^{N+2s}}
\,\dx\,\dy\,\dt\\
& \leq
\gamma 2^{(N+2s)n}
|\check B_n|
\int_0^\tau
\int_{\hat B_n^c}
\frac{u_+(y,t)}
{|y|^{N+2s}}
\,\dy\,\dt\\
& \leq
\gamma b^n\rho^N
\textnormal{Tail}
\bigl(u_+;B_\rho  \times(0,\tau]\bigr).
\end{align*}
Consequently, when $t_1>t_2$,
\begin{align} \label{k1}
    J_1'
    \leq
    \epsilon S_{n+1}
    +
    \gamma_\epsilon b^n
    \Bigl(\frac{\tau}{\rho^\lambda}\Bigr)^{\frac q{q-1}}
    \left(
    P_-+P_-^{\frac q{q-1}}
    \right)
    +\gamma b^n\rho^N
    \textnormal{Tail}
    \bigl(u_+;B_\rho  \times(0,\tau]\bigr),
\end{align}
for some $\epsilon\in(0,1)$ to be chosen later, with
$\gamma_\epsilon\to\infty$ as $\epsilon\to0$ and $b>1$ depending
only on the data. Therefore, from
\eqref{ineq:prelim_estimate_of_S_n_+_1} and \eqref{forwardJ1}, \eqref{k1}, we get the
recursive inequalities
\begin{align}\label{ineq:iteration_inequality_for_S_n_+_1}
    S_n
    &\leq
    I+\epsilon S_{n+1}
    +
    \gamma_\epsilon b^n
    \Bigl(\frac{\tau}{\rho^\lambda}\Bigr)^{\frac q{q-1}}
    \left(
    P_-+P_-^{\frac q{q-1}}
    \right)
    +
    \gamma b^n\rho^N
    \textnormal{Tail}
\bigl(u_+;B_\rho  \times(0,\tau]\bigr) \notag  \\ & \leq 
    \epsilon S_{n+1}
    +
    \gamma_\epsilon b^n
    \Bigg[
    I+
    \Bigl(\frac{\tau}{\rho^\lambda}\Bigr)^{\frac q{q-1}}
    \left(
    P_-+P_-^{\frac q{q-1}}
    \right)
    +
    \rho^N
    \textnormal{Tail}
    \bigl(u_+;B_\rho  \times(0,\tau]\bigr)
    \Bigg].
\end{align}
By iterating \eqref{ineq:iteration_inequality_for_S_n_+_1}, we
obtain
\begin{align*}
    S_0
    \leq&
    \epsilon^nS_n
    +
    \gamma_\epsilon
    \Bigg[
    I+
    \Bigl(\frac{\tau}{\rho^\lambda}\Bigr)^{\frac q{q-1}}
    \left(
    P_-+P_-^{\frac q{q-1}}
    \right)
    +
    \rho^N
    \textnormal{Tail}
    \bigl(u_+;B_\rho  \times(0,\tau]\bigr)
    \Bigg]
    \sum_{i=0}^{n-1}(\epsilon b)^i.
\end{align*}
At this point, we choose
\(
\epsilon\in\left(0,\frac1{2b}\right)
\)
so the series on the right-hand side is bounded by
$2$. Letting $n\to\infty$, we have
\begin{align*}
    S_0
    \leq
    \gamma
    \Bigg[
    I+
    \Bigl(\frac{\tau}{\rho^\lambda}\Bigr)^{\frac q{q-1}}
    \left(
    P_-+P_-^{\frac q{q-1}}
    \right)
    +
    \rho^N
    \textnormal{Tail}
    \bigl(u_+;B_\rho  \times(0,\tau]\bigr)
    \Bigg],
\end{align*}
which implies \eqref{prop3.2}.
\end{proof}

\noindent
We can now proceed to prove  Theorem \ref{lq-linf}.

\begin{proof}[Proof of Theorem \ref{lq-linf}]
   Take $(x_0,t_0)=(0,0)$ for simplicity and set $P_{-}=P_{-}(\rho, x_0; t_0 -\sigma, t_0)$. Using Proposition \ref{lq-lq} and dividing   both sides by $\rho^N$, we have
    \begin{equation*}
        \begin{aligned}
             \sup_{t \in [-\sigma,0]} \fint_{B_{\rho}  } u^{q}(x,t) \,\dx \leq \gamma \inf_{t \in [-\sigma,0]} \fint_{B_{2 \rho}  } u^{q}(x,t) \,\dx + \frac{\gamma}{\rho^N} \Bigl( \frac{\sigma}{\rho^{\lambda}} \Bigl)^{\frac{q}{q-1}} \big( P_{-} + P_{-}^{\frac{q}{q-1}} \big) + \gamma\operatorname{Tail}\bigl(u_+;B_\rho\times(-\sigma,0]\bigr) \, ,
        \end{aligned}
    \end{equation*} Then, applying  \eqref{tesi1} with $\texttt{r}=q$ and the estimate above, we get
   \begin{equation*}
        \begin{aligned}
      \esssup_{Q_{  \varrho/2,  \sigma/2}} u \leq & \gamma \Big[ \fiint_{B_{\rho} \times (-\sigma,0)} u^q(x,t) \,\dx\,\dt \Big]^{\frac{2s}{\Lambda_q}}\Big( \frac{\sigma}{\rho^{2s}} \Big)^{-\frac{N}{\Lambda_q}} + \gamma \Big( \frac{\sigma}{\rho^{2s}} \Big)^{\frac{1}{q-1}} + \gamma \textnormal{Tail}^{\frac{1}{q}}\bigl(u;B_{\rho/2}\times(-\sigma,0]\bigr) \\  
       \leq & \gamma \Big[ \inf_{t \in [-\sigma,0]} \fint_{B_{2 \rho}} u^{q}(x,t) \,\dx + \frac{\gamma}{\rho^N} \left( \frac{\sigma}{\rho^{\lambda}} \right)^{\frac{q}{q-1}} \big( P_{-} + P_{-}^{\frac{q}{q-1}} \big) + \gamma\operatorname{Tail}\bigl(u_+;B_\rho\times(-\sigma,0]\bigr) \Big]^{\frac{2s}{\Lambda_q}}\Big( \frac{\sigma}{\rho^{2s}} \Big)^{-\frac{N}{\Lambda_q}} \\
      & + \gamma \Big( \frac{\sigma}{\rho^{2s}} \Big)^{\frac{1}{q-1}} + \gamma \textnormal{Tail}^{\frac{1}{q}}\bigl(u;B_{\rho/2}\times(-\sigma,0]\bigr) \\
      \leq & \gamma \Big[ \inf_{t \in [-\sigma,0]} \fint_{B_{2 \rho}} u^{q}(x,t) \,\dx \Big]^{\frac{2s}{\Lambda_q}}\Big( \frac{\sigma}{\rho^{2s}} \Big)^{-\frac{N}{\Lambda_q}} + \gamma \Big( \frac{\sigma}{\rho^{2s}} \Big)^{\frac{q}{q-1} (\frac{2s}{\Lambda_q})- \frac{N}{\Lambda_q}} \Big[\big( P_{-} + P_{-}^{\frac{q}{q-1}} \big) \Big]^{\frac{2s}{\Lambda_q}} \\
      & + \gamma \Big( \frac{\sigma}{\rho^{2s}} \Big)^{-\frac{N}{\Lambda_q}}\textnormal{Tail}^{\frac{2s}{\Lambda_q}}\bigl(u_+;B_{\rho}\times(-\sigma,0]\bigr)
      + \gamma \Big( \frac{\sigma}{\rho^{2s}} \Big)^{\frac{1}{q-1}} + \gamma \textnormal{Tail}^{\frac{1}{q}}\bigl(u;B_{\rho/2}\times(-\sigma,0]\bigr) \\
      \leq  & \gamma \Big[ \inf_{t \in [-\sigma,0]} \int_{B_{2 \rho}} u^{q}(x,t) \,\dx \Big]^{\frac{2s}{\Lambda_q}} \sigma^{-\frac{N}{\Lambda_q}} + \gamma \Big( \frac{\sigma}{\rho^{2s}} \Big)^{\frac{1}{q-1}} \Big[\big( P_{-} + P_{-}^{\frac{q}{q-1}} \big) \Big]^{\frac{2s}{\Lambda_q}} \\
      & + \gamma \Big( \frac{\sigma}{\rho^{2s}} \Big)^{-\frac{N}{\Lambda_q}}\textnormal{Tail}^{\frac{2s}{\Lambda_q}}\bigl(u_+;B_{\rho}\times(-\sigma,0]\bigr)
      + \gamma \Big( \frac{\sigma}{\rho^{2s}} \Big)^{\frac{1}{q-1}} + \gamma \textnormal{Tail}^{\frac{1}{q}}\bigl(u_+;B_{\rho/2}\times(-\sigma,0]\bigr).
        \end{aligned}
    \end{equation*}
   Note that in the last inequality we used
    \[
\operatorname{Tail}^{\frac1q}\!\left(u_-;B_{\rho/2}\times(-\sigma,0]\right)
\le
\left(\frac{\sigma}{\rho^{2s}}\right)^{\frac{1}{q-1}}
\left(
P_-+P_-^{\frac{q}{q-1}}
\right)^{\frac{2s}{\Lambda_q}}.
\] Note that, using \(q>1\) and the explicit expression for \(P_-\), we obtain \eqref{intH}. The proof is complete.
\end{proof}
\section{Expansion of positivity} \label{sec4}
\noindent
In this section we prove Theorem \ref{prop:expansion_positivity}. 
Compared with \cite{MY}, the argument below yields the expansion of positivity under an $L^1$ formulation of the  tail. The proof is based on a sequence of auxiliary lemmas collected in the section below.
\subsection{Preliminary tools}
In the following, let \(Q:=B_{R}(x_0)\times(t_1,t_2]\Subset\Omega_T\). The first lemma is a De Giorgi type result.
\begin{lemma}[A De Giorgi type lemma]\label{lem1}
Let \(u\) be a local, nonnegative in $Q$, weak super-solution to~\eqref{eq:PDE} in \(\Omega_T\). For some \(\delta\in(0,1)\) and \(k>0\), set \(\theta=\delta k^{q-1}\) and assume \(Q_{2\varrho}(\theta):= B_{2\rho}(x_0)\times (t_0-\theta\rho^{2s},t_0)\subset Q\). There exists $\Theta \equiv \Theta (N,s,q,\texttt{L}) > 1$ and \(\nu\in(0,1)\), depending only on the data and \(\delta\), such that if
\[
\bigl|\{u\leq k\}\cap Q_{\varrho}(\theta)\bigr|\leq \nu |Q_{\varrho}(\theta)|,
\]
then either
\[
\Theta \, \textnormal{Tail}^{\frac{1}{q}}\bigl(u_{-};Q\bigr)> \frac{k}{4},
\]
or
\[
u\geq \frac{k}{4}\qquad\text{almost everywhere in }Q_{\frac\varrho 2}(\theta).
\]
Moreover, we have the dependence \(\nu\approx\delta^\beta\), for some \(\beta>1\) depending only on \(N,s\) and \(q\).
\end{lemma}
\begin{proof}
Assume $(x_0, t_0)=(0,0)$. As a restatement, we show that there exists \(\nu\in(0,1)\), depending only on the data and \(\delta\), such that if
\begin{equation} \label{fta}
\Theta \, 
\textnormal{Tail}^{\frac{1}{q}}\bigl(u_-;Q\bigr)\leq\frac{k}{4},
\end{equation}
and
\begin{equation} \label{t2}
    \bigl|\{u\leq k\}\cap Q_{\varrho}(\theta)\bigr|\leq \nu |Q_{\varrho}(\theta)|,
\end{equation}
then
\[
u\geq \frac{k}{4}\qquad\text{almost everywhere in }Q_{\frac\varrho2}(\theta).
\]
 We define
\begin{equation} \label{elll}
    \ell(t):= \Theta
\left[
\int_{t_1}^{t}\int_{\R^N\setminus B_{R}}
\frac{(u(y,\tau))_-}{|y|^{N+2s}}\,\dy\,\d\tau
\right]^{\frac1q},
\end{equation}
for \(t\in(t_1,t_2]\). Then \(\ell\) is nondecreasing and, by  \eqref{fta},
\begin{equation} \label{el1}
    0\leq \ell(t)\leq \ell(t_2)\leq \frac{k}{4}
\qquad\text{for every }t\in(t_1,t_2].
\end{equation}
\noindent
For \(n\in\N_0\), set
\[
\varrho_n:=\frac{\varrho}{2}+\frac{\varrho}{2^{n+2}},
\quad
\sigma_n:=\theta\varrho_n^{2s},
\quad
B_n:=B_{\varrho_n},
\quad
Q_n:=B_n\times(-\sigma_n,0),
\quad
k_n:=\frac{k}{2}+\frac{k}{2^{n+1}},
\quad
k_n(t):=k_n-\ell(t),
\]
\[
\check\varrho_n:=\frac{3\varrho_n+\varrho_{n+1}}{4},
\quad
\check\sigma_n:=\frac{3\sigma_n+\sigma_{n+1}}{4},
\quad
\check B_n:=B_{\check\varrho_n},
\quad
\check Q_n:=\check B_n\times(-\check\sigma_n,0),
\]
\[
\tilde\varrho_n:=\frac{\varrho_n+\varrho_{n+1}}{2},
\quad
\tilde\sigma_n:=\theta\tilde\varrho_n^{2s},
\quad
\tilde B_n:=B_{\tilde\varrho_n},
\quad
\tilde Q_n:=\tilde B_n\times(-\tilde\sigma_n,0),
\]
and
\[
\hat\varrho_n:=\frac{\varrho_n+3\varrho_{n+1}}{4},
\quad
\hat\sigma_n:=\frac{\sigma_n+3\sigma_{n+1}}{4},
\quad
\hat B_n:=B_{\hat\varrho_n},
\quad
\hat Q_n:=\hat B_n\times(-\hat\sigma_n,0).
\]
Then
\[
Q_{n+1}\subsetneq \hat Q_n\subsetneq \tilde Q_n
\subsetneq \check Q_n\subsetneq Q_n \subsetneq Q_{\varrho}(\theta).
\]
We introduce the truncation
\[
w_-(x,t):=\bigl(u(x,t)-k_n(t)\bigr)_-.
\]
By  \eqref{el1}, we have
\begin{equation}\label{w-}
k_n(t)\geq \frac{k}{2}-\frac{k}{4}=\frac{k}{4},
\qquad
k_n(t)\leq k_n\leq k,
\qquad
0\leq w_-\leq k \text{ in } Q.
\end{equation}
Let
\[
A_n:=\{(x,t)\in Q_n:u(x,t)<k_n(t)\}.
\]
We apply  \eqref{cacc1} for supersolutions in $Q_n$, with  $k(t)=k_n(t)$ and with a cut-off function \(0\leq\zeta\leq1\), \(\zeta=1\) in \(\tilde Q_n\),  \(\operatorname{supp}\zeta\subset \check Q_n\) and such that
\begin{align} \label{suzeta}
    |D\zeta|\lesssim \frac{2^n}{\varrho}, \quad  |\partial_t\zeta|\lesssim \frac{2^n}{\theta\varrho^{2s}}.
\end{align}
Now, we drop the mixed term on the left, while we treat the first two terms on the right-hand side as in the proof of \cite[Lemma 3.1]{L4}. More precisely, to handle the second term, using \(w_-(y,t)\leq u_-(y,t)+k\), the contribution over \(\mathbb R^N\setminus B_{R}\) involving \(u_-(y,t)\) is absorbed by the term containing \((k_n^q(t))'\), owing to the definition of \(\ell\) and the choice of
\begin{align}\label{def:Theta}
\Theta := (q 2^{N+2s+2}\texttt{L})^{\frac1q}. 
\end{align}
Indeed, by \eqref{el1} and \eqref{w-}, we have \(k_n(t)\geq \ell(t)\), and hence
\[
- \frac{1}{q} \bigl(k_n^q(t)\bigr)'
= k_n^{q-1}(t)\ell'(t)
\geq \ell^{q-1}(t)\ell'(t)
=\frac{\Theta^q}{q}
\int_{\mathbb R^N\setminus B_{R}}
\frac{u_-(y,t)}{|y|^{N+2s}}\,\dy.
\] Therefore, we obtain
\begin{align}\label{ineq:caccioppoli_dg_tail_super}
&\esssup_{-\tilde\sigma_n<t<0}
\int_{\tilde B_n}\mathfrak{g}_-\bigl(u(x,t),k_n(t)\bigr)\,\dx
+   \frac{1}{\texttt{L}}\int_{-\tilde\sigma_n}^{0}
\iint_{\tilde B_n\times \tilde B_n}
\frac{|w_-(x,t)-w_-(y,t)|^2}{|x-y|^{N+2s}}\,\dx\,\dy\,\dt
\notag\\
&\leq \gamma\,\frac{2^n}{\theta\varrho^{2s}}
\iint_{Q_n}\mathfrak{g}_-\bigl(u(x,t),k_n(t)\bigr)\,\dx\,\dt \notag \\ & \qquad +\gamma
\int_{-\sigma_n}^{0}
\iint_{B_n\times B_n}
\max\left\{
w_-^2(x,t),w_-^2(y,t)
\right\}
\frac{|\zeta(x,t)-\zeta(y,t)|^2}
{|x-y|^{N+2s}}
\,\dx\,\dy\,\dt .
\notag\\
&\qquad
+\gamma\,2^{n(N+2s)}\frac{1}{\varrho^{2s}}
\left[
\esssup_{t_1<t<t_2}
\fint_{B_{2\varrho}}w_-(y,t)\,\dy
\right]
\iint_{Q_n}w_-(x,t)\,\dx\,\dt
+\gamma\frac{k}{\varrho^{2s}}\iint_{Q_n}w_-(x,t)\,\dx\,\dt
\notag\\
&=:I_1+I_2+I_3+I_4.
\end{align}
We estimate $I_1,I_2,I_3, I_4$ on the right-hand side of \eqref{ineq:caccioppoli_dg_tail_super}. On \(A_n\), by  \eqref{GG}, we have
\begin{equation} \label{suG}
    \frac{1}{q+1} w_{-}^{q+1}(x,t) \leq \mathfrak{g}_-\bigl(u(x,t),k_n(t)\bigr) \leq \frac{1}{q+1} k^{q+1}.
\end{equation}
Therefore, using  \eqref{suG},  \eqref{w-} and \(\theta=\delta k^{q-1}\),
\begin{equation}\label{I_1}
I_1 = \gamma\,\frac{2^n}{\theta\varrho^{2s}}
\iint_{A_n}\mathfrak{g}_-\bigl(u(x,t),k_n(t)\bigr)\,\dx\,\dt
\leq
\gamma\,2^n\frac{k^2}{\delta\varrho^{2s}}|A_n|.
\end{equation}
For the second term, using \eqref{suzeta} and  \eqref{w-}, we get
\begin{equation}\label{I_2}
I_2
\leq \gamma \,\frac{2^{2n}}{\varrho^{2s}}\iint_{Q_n}w_-^2(x,t)\,\dx\,\dt \leq 
\gamma\,2^{2n}\frac{k^2}{\varrho^{2s}}|A_n|.
\end{equation}
For $I_3$ and $I_4$, using again  \eqref{w-}, we get
\[
\esssup_{t_1<t<t_2}\fint_{B_{2\varrho}}w_-(y,t)\,\dy\leq k,
\qquad
\iint_{Q_n}w_-(x,t)\,\dx\,\dt\leq k|A_n|,
\]
therefore
\begin{equation}\label{I_3}
I_3 + I_4
\leq
\gamma\,2^{n(N+2s)}\frac{k^2}{\varrho^{2s}}|A_n|.
\end{equation}
Combining  \eqref{I_1}-\eqref{I_3} with  \eqref{ineq:caccioppoli_dg_tail_super}, using \(\delta\in(0,1)\), we obtain
\begin{align}\label{ineq:energy_estimate_dg_iteration_super}
&\esssup_{-\tilde\sigma_n<t<0}
\int_{\tilde B_n}\mathfrak{g}_-\bigl(u(x,t),k_n(t)\bigr)\,\dx
+\int_{-\tilde\sigma_n}^{0}
\iint_{\tilde B_n\times \tilde B_n}
\frac{|w_-(x,t)-w_-(y,t)|^2}{|x-y|^{N+2s}}\,\dx\,\dy\,\dt
 \leq
\gamma\,b^n\frac{k^2}{\delta\varrho^{2s}}|A_n|,
\end{align}
for some \(b>1\) depending only on \(N\) and \(s\).

Consider now a cut-off function \(0\leq \phi\leq1\) on \(\tilde Q_n\),
which vanishes outside \(\hat Q_n\), satisfies \(\phi=1\) in \(Q_{n+1}\), and
\(
|D\phi|\leq {2^{n+5}}/{\varrho}.
\)
By the parabolic Sobolev embedding from Lemma~\ref{em}, with \(d=2^{-n-6}\), we get
\begin{align} \label{4.9}
&
\iint_{\tilde Q_n}
\bigl[\phi(x,t)w_-(x,t)\bigr]^\mm\,\dx\,\dt
\leq
\gamma\biggl[
\varrho^{2s}
\int_{-\tilde\sigma_n}^{0}
\iint_{\tilde B_n\times\tilde B_n}
\frac{
\left|
\phi(x,t)w_-(x,t)
-
\phi(y,t)w_-(y,t)
\right|^2
}{|x-y|^{N+2s}}
\,\dx\,\dy\,\dt
\notag\\
&\quad
+2^{n(N+2s)}
\iint_{\tilde Q_n}
\bigl[\phi(x,t)w_-(x,t)\bigr]^2\,\dx\,\dt
\biggr]
\biggl[
\esssup_{-\tilde\sigma_n<t<0}
\fint_{\tilde B_n}
\bigl[\phi(x,t)w_-(x,t)\bigr]^{q+1}\,\dx
\biggr]^{\frac{2s}{N}}=:\gamma\bigl(R_1+R_2\bigr)R_3^{\frac{2s}{N}},
\end{align}
where we used $\tilde{\rho}_n \approx \rho$. We now estimate the three terms in the right-hand side above. 
Using
\[
\begin{aligned}
&\left|
\phi(x,t)w_-(x,t)
-
\phi(y,t)w_-(y,t)
\right|^2
\\
& \qquad\leq
2\left|
w_-(x,t)
-
w_-(y,t)
\right|^2\phi^2(x,t)
+2w_-^2(y,t)
|\phi(x,t)-\phi(y,t)|^2,
\end{aligned}
\]
and  \eqref{w-},  \eqref{ineq:energy_estimate_dg_iteration_super}, we get
\begin{align}
R_1 + R_2
&\leq
\gamma\varrho^{2s}
\int_{-\tilde\sigma_n}^{0}
\iint_{\tilde B_n\times\tilde B_n}
\frac{|w_-(x,t)-w_-(y,t)|^2}{|x-y|^{N+2s}}
\,\dx\,\dy\,\dt
+\gamma\,2^{2n}
\iint_{\tilde Q_n}w_-^2(x,t)\,\dx\,\dt \notag \\ & \qquad + 2^{n(N+2s)}
\iint_{\tilde Q_n\cap A_n}
w_-^2(x,t)\,\dx\,\dt \leq 
\gamma\,b^n\frac{k^2}{\delta}|A_n|  +  b^n k^2 |A_n|
\leq
\gamma\,b^n\frac{k^2}{\delta}|A_n|,
\label{eq:R1_estimate}
\end{align}
with $b>1$ depending on $N,s$. As regards \(R_3\), by  \eqref{ineq:energy_estimate_dg_iteration_super},  \eqref{suG},  \eqref{w-}, it holds
\begin{align}
R_3
\leq
\gamma
\esssup_{-\tilde\sigma_n<t<0}
\fint_{\tilde B_n}
\mathfrak{g}_-\bigl(u(x,t),k_n(t)\bigr)
\,\dx
\leq
\gamma\,b^n
\frac{k^2}{\delta\varrho^{N+2s}}
|A_n|.
\label{eq:R3_estimate}
\end{align}
Combining  \eqref{eq:R1_estimate}, \eqref{eq:R3_estimate} with  \eqref{4.9}, we obtain
\begin{align}\label{eq:Lm_estimate}
\iint_{\tilde Q_n}
\bigl[\phi(x,t)w_-(x,t)\bigr]^\mm\,\dx\,\dt \leq 
\gamma\,b^n\frac{k^2}{\delta}|A_n|
\left(
b^n\frac{k^2}{\delta\varrho^{N+2s}}|A_n|
\right)^{\frac{2s}{N}}.
\end{align}
Since \(\phi=1\) in \(Q_{n+1}\), we have 
\begin{equation} \label{a1}
    \iint_{A_{n+1}}w_-\,\dx\,\dt \leq \iint_{\tilde Q_n} \bigl[\phi(x,t)w_-(x,t)\bigr]\,\dx\,\dt.
\end{equation}
Moreover, on \(A_{n+1}\) we have 
\begin{equation} \label{a2}
    w_-(x,t) = k_n(t)-u(x,t) \geq k_n(t)-k_{n+1}(t) = k_n-k_{n+1} = \frac{k}{2^{n+2}}.
\end{equation}
Therefore, by H\"older's inequality with exponents $\left (\mm,\frac{\mm}{\mm-1} \right )$,  \eqref{a2},  \eqref{a1} and  \eqref{eq:Lm_estimate}, we get
\begin{align} \label{pe} \frac{k}{2^{n+2}} |A_{n+1}| &\leq \iint_{A_{n+1}}w_-\,\dx\,\dt \leq \left (\iint_{\tilde{Q}_n}(w_-\phi)^\mm\,\dx\,\dt \right)^{\frac{1}{\mm}}|A_{n}|^{\frac{\mm-1}{\mm}} \leq \gamma\,b^n\frac{k^{\left(2+\frac{4s}{N}\right)\frac{1}{\mm}}|A_n|^{1+\frac{2s}{N\mm} }}{\delta^{\frac{N+2s}{N\mm}}\varrho^{(N+2s)\frac{2s}{N\mm}}} . \end{align} 
Define now
\[
Y_n:=\frac{|A_n|}{|Q_n|}.
\]
From  \eqref{pe} and the definition of $\mm$ in  \eqref{m},
we obtain
\[
|A_{n+1}|
\leq
\gamma\,b^n
\frac{1}
{\delta^{\frac{N+2s}{N\mm}}k^{(q-1)\frac{2s}{N\mm}}\varrho^{(N+2s)\frac{2s}{N\mm}}}
|A_n|^{1+\frac{2s}{N\mm}}.
\]
Dividing by \(|Q_{n+1}|\) and using \(\rho_n \approx \rho\), we infer
\[
Y_{n+1}
\leq
\gamma\,b^n\delta^{-\frac{1}{\mm}}Y_n^{1+\frac{2s}{N\mm}}.
\]
By the fast geometric convergence Lemma \ref{fg},  if \(Y_0\leq\nu_0\), with \((0,1) \ni \nu_0\approx\delta^{\frac{N}{2s}}\), then \(Y_n\to0\).
Now, since \(Q_0=Q_{3\varrho/4}(\theta)\) and \(k_0(t)=k-\ell(t)\leq k\), we have
\[
A_0\subset \{u< k\}\cap Q_{\varrho}(\theta)
\]
and, by  \eqref{t2}, it holds
\[
Y_0\leq \left (\frac{4}{3}\right )^{N+2s}\nu,
\]
therefore, choosing $\nu \leq \left (\frac{3}{4}\right )^{N+2s}\nu_0$,
we obtain \(Y_n\to0\). Since \(\varrho_n\to\varrho/2\) and \(k_n(t)\to k/2-\ell(t)\), we get
\[
\left|
\left\{(x,t)\in Q_{\frac{\varrho}{2}}(\theta):u(x,t)<\frac{k}{2}-\ell(t)\right\}
\right|=0.
\]
Thus
\[
u\geq \frac{k}{2}-\ell(t)
\qquad\text{almost everywhere in }Q_{\frac{\varrho}{2}}(\theta).
\]
By  \eqref{el1}, we obtain
\[
u\geq \frac{k}{4}
\qquad\text{almost everywhere in }Q_{\frac{\varrho}{2}}(\theta).
\]
This ends the proof.
\end{proof}

\noindent
The next lemma shows that measure-theoretic information propagates forward in time.

\begin{lemma}\label{lem3}
Let $q>1$, and let $u$ be a local, nonnegative in $Q$, weak
super-solution to \eqref{eq:PDE} in $\Omega_T$. Let $\alpha\in(0,1)$.
Then there exist  $\delta,\varepsilon, \eta_0\in(0,1)$, depending only on the data
and on $\alpha$, such that, if
\[
\left|
\left\{
u(\cdot,t_0)\geq k
\right\}
\cap B_\varrho(x_0)
\right|
\ge
\alpha |B_\varrho|,
\]
for some $k>0$, then either
\[
 \textnormal{Tail}^{\frac{1}{q}}\bigl(u_{-};Q\bigr)
>
\eta_0k,
\]
or
\[
\left|
\left\{
u(\cdot,t)\geq \varepsilon k
\right\}
\cap B_\varrho(x_0)
\right|
\geq
\frac{\alpha}{2}|B_\varrho|, \qquad
\text{for every }
t\in
\left(
t_0,t_0+\delta k^{q-1}\varrho^{2s}
\right],
\]
provided the cylinders involved are included in $Q$. Moreover, $\eta_0^q \lesssim \delta \approx \alpha^{N+2s+1}$ and $\varepsilon \approx \alpha$.
\end{lemma}
\begin{proof} We may assume, without loss of generality, that
$(x_0,t_0)=(0,0)$. We show that there exist $\delta,\varepsilon, \eta_0\in(0,1)$,
depending only on the data and on $\alpha$, such that if
\begin{equation} \label{hh11}
    \left|\{u(\cdot,0)\geq k\}\cap B_\varrho\right|
\ge
\alpha |B_\varrho|,
\end{equation}
and
\begin{equation} \label{eta0}
\textnormal{Tail}^{\frac{1}{q}}(u_-;Q)
\leq \eta_0k,
\end{equation}
then
\begin{equation} \label{res}
    \left|\{u(\cdot,t) \geq \varepsilon k\}\cap B_\varrho\right|
\ge
\frac{\alpha}{2}|B_\varrho|, \qquad \text{for every } t\in \left(0,\delta k^{q-1}\varrho^{2s}\right].
\end{equation}
We apply \eqref{cacc2} for supersolution in the cylinder $Q_{R,S}:= B_{\rho} \times (0,\delta k^{q-1}\varrho^{2s}) $, with  $k(t)=k$, and with a cut-off function \(0\leq\xi\leq1\), \(\xi=1\) in \(B_{(1-\sigma)\varrho}\),  \(\operatorname{supp}\xi\subset B_{(1-\sigma/2)\varrho}\) and such that
\begin{align} \label{suxi}
    |D\xi|\lesssim \frac{1}{\sigma\varrho}.
\end{align}
 Then, defining $w_{-}(x,t):= (u(x,t)-k)_{-}$ and keeping just the first term on the left hand-side, we arrive at
\begin{align} \label{rlh}
& \int_{B_{(1-\sigma)\varrho}}
\mathfrak{g}_-(u(x,\tau),k)\,\dx
 \leq 
\int_{B_\varrho\times\{0\}}
\mathfrak{g}_-(u(x,0),k)\,\dx  \\ & \qquad + \gamma
\int_0^{\delta k^{q-1}\rho^{2s}}
\iint_{B_\varrho \times B_\varrho}
\max\{w_-^2(x,t),w_-^2(y,t)\}
\frac{|\xi(x)-\xi(y)|^2}
{|x-y|^{N+2s}}
\,\dx\,\dy\,\dt
\notag \\
&\qquad+
4\texttt{L}
\int_0^{\delta k^{q-1}\varrho^{2s}}
\left(
\int_{B_\varrho}
\xi^2(x)w_-(x,t)\,\dx
\right)
\left(
\operatorname*{ess\,sup}_{x\in B_{(1-\sigma/2)\varrho}}
\int_{\mathbb R^N\setminus B_\varrho}
\frac{w_-(y,t)}
{|x-y|^{N+2s}}
\,\dy
\right)
\dt . \notag \\ & =: I_1 + I_2 + I_3 , \notag
\end{align}
for any $\tau \in (0, \delta k^{q-1}\rho^{2s}]$.
We estimate $I_1,I_2,I_3$. For $I_1$, using  \eqref{suG} and  \eqref{hh11}, we get
\begin{equation} \label{forI1}
    I_1 \leq \frac{1}{q+1} |\{x \in B_\rho: u(x,0) < k\}|k^{q+1} \leq \frac{1}{q+1} (1-\alpha) k^{q+1}|B_\varrho|.
\end{equation}
For $I_2$, using \eqref{suxi} and \(\sigma\in(0,1)\), we obtain
\begin{equation} \label{I12}
     I_2 \leq \frac{\gamma}{\sigma^2 \rho^{2s}} \int_0^{\delta k^{q-1}\varrho^{2s}} \int_{B_\varrho} w^2_{-}(x,t) \, \dx \, \dt
\leq 
\frac{\gamma \delta}{\sigma^{N+2s}} k^{q+1}|B_\varrho|.
\end{equation}
For $I_3$ we observe that, by  $w_{-}(y,t) \leq k + u_{-}(y,t)$, the estimate
\[
|x-y|
\geq |y|-|x|
\geq \frac{\sigma}{2}|y|,
\qquad \text{for }
x\in B_{(1-\sigma/2)\rho} \text{ and }
y\in\mathbb{R}^N\setminus B_\rho,
\]
and inequality \eqref{eta0}, we have
\begin{align} \label{3i3}
I_3 &\leq \gamma k |B_\rho| \left ( \int_0^{\delta k^{q-1}\rho^{2s}} \operatorname*{ess\,sup}_{x\in B_{(1-\sigma/2)\varrho}} \left (
\int_{\mathbb R^N\setminus B_\varrho}
\frac{k}
{|x-y|^{N+2s}}
\,\dy  + 
\int_{\mathbb R^N\setminus B_R}
\frac{u_{-}(y,t)}
{|x-y|^{N+2s}}
\,\dy \right ) \, \dt\right )
\notag \\ &\le
\frac{\gamma}{\sigma^{N+2s}} k|B_\varrho|
\left(
\delta k^q
+
\operatorname{Tail}(u_-;Q)
\right)
\notag \\
&\leq \frac{\gamma}{\sigma^{N+2s}} 
k|B_\varrho|
\left(
\delta k^q
+
\eta_0^q k^q
\right)
\notag \\
&=
\left(
\delta
+
\eta_0^q
\right) \frac{\gamma}{\sigma^{N+2s}}
k^{q+1}|B_\varrho|,
\end{align}
note that we used also the nonnegativity of the solution in $Q$.
Let us now estimate from below the first term on the left-hand side of  \eqref{rlh}. Using  \eqref{suG}, we get
\begin{align} \label{sur}
     \int_{B_{(1-\sigma)\varrho}} \mathfrak{g}_-(u(x,\tau),k)\,\dx  & \geq \frac{1}{q+1}\int_{B_{(1-\sigma)\rho}} (u(x,\tau)-k)_{-}^{q+1} \chi_{ \{u(\cdot,\tau) <\varepsilon k  \}}\,\dx \notag \\ & \geq  \frac{\left(1 -\varepsilon \right )^{q+1}}{q+1} k^{q+1}|A_{\varepsilon,(1-\sigma)\rho} |  ,
\end{align}
where
$$ A_{\varepsilon,(1-\sigma)\rho} 
:=
\left\{
x\in B_{(1-\sigma)\rho}: u(x,\tau)<\varepsilon k 
\right\}.$$
Note that,
\begin{align} \label{aep}
|A_{\varepsilon,\rho} |
&=
\left|
A_{\varepsilon,(1-\sigma)\rho} 
\cup
\left(
A_{\varepsilon,\rho} \setminus A_{\varepsilon,(1-\sigma)\rho} 
\right)
\right| \notag \\
&\le
|A_{\varepsilon,(1-\sigma)\rho} |
+
|B_\rho\setminus B_{(1-\sigma)\rho}|\notag  \\
&\le
|A_{\varepsilon,(1-\sigma)\rho} |
+
N\sigma |B_\rho|.
\end{align}
Therefore, choosing 
$$ \eta_0^q \lesssim \delta$$
and using  \eqref{forI1},  \eqref{I12}, \eqref{3i3},  \eqref{sur},  \eqref{aep} in  \eqref{rlh}, we get
\begin{align}
    |A_{\varepsilon,\rho} | \leq  \frac{1}{\left(1-\varepsilon\right)^{q+1}} \left (\gamma \frac{\delta}{\sigma^{N+2s}} +(1-\alpha)\right )|B_\rho|  + N\sigma |B_\rho|.
\end{align}
Now, we choose $\sigma, \delta, \varepsilon$ such that
\begin{equation*}
    N\sigma=\frac{1}{4}\alpha, \qquad \gamma\frac{\delta}{\sigma^{N+2s}}=  \frac{1}{8} \alpha, \qquad  \varepsilon  = 1- \left (  \frac{1-\frac{7}{8}\alpha}{1-\frac{3}{4}\alpha} \right )^{\frac{1}{q+1}},
\end{equation*}
In this way, we obtain
$$ |A_{\varepsilon,\rho} |\leq \left ( 1-\frac{1}{2}\alpha\right )|B_\rho|,$$
for all $0<t<\delta k^{q-1}\rho^{2s}$, that is  \eqref{res}.
\end{proof}
\noindent
The following lemma is a measure shrinking result.
\begin{lemma}\label{lem4}
Let $q>1$, and let $u$ be a local, nonnegative in $Q$, weak super-solution to \eqref{eq:PDE}
in $\Omega_T$. Assume that for some $\alpha\in(0,1)$, $\delta,\sigma\in(0,1/2)$ and $k>0$, there holds
\[
\bigl|\{u(\cdot,t)\geq k\}\cap B_\varrho(x_0)\bigr|
\geq \alpha |B_\varrho|, \qquad \text{for every } t\in \bigl(t_0-\delta(\sigma k)^{q-1}\varrho^{2s},t_0\bigr].
\]
Let $\theta:=\delta(\sigma k)^{q-1}$ and  \(Q_{\varrho}(\theta):= B_\rho(x_0)\times (t_0-\theta\rho^{2s},t_0)\subset Q\). Then there exists  a constant $\gamma>0$, depending only on the data and
independent of $\alpha,\delta,\sigma,k$, such that either
\[\textnormal{Tail}^{\frac{1}{q}}(u_{-};Q)
> \frac{\sigma k}{4},
\]
or
\[
\bigl|\{u\leq \sigma k\}\cap Q_\varrho(\theta)\bigr|
\leq
\frac{\gamma\sigma}{\delta\alpha}\,|Q_\varrho(\theta)|,
\]
provided $Q_{2\rho}(\theta):= B_{2\rho} \times (t_0-\theta\rho^{2s},t_0) \subset Q$.
\end{lemma}

\begin{proof}
We may assume, without loss of generality, that $(x_0,t_0)=(0,0)$. We show that, if
\begin{equation}\label{ms1}
|\{u(\cdot,t)\geq k\}\cap B_\varrho|\geq \alpha |B_\varrho|, \qquad \text{for every } t\in \bigl(-\delta(\sigma k)^{q-1}\varrho^{2s},0\bigr],
\end{equation}
and
\begin{equation}\label{ms2}
 \textnormal{Tail}^{\frac{1}{q}}(u_-;Q)\leq \frac{\sigma k}{4},
\end{equation}
then
\[
|\{u\leq \sigma k\}\cap Q_\varrho(\theta)|
\leq
\frac{\gamma\sigma}{\delta\alpha}\,|Q_\varrho(\theta)|.
\]
We apply \eqref{cacc2} in $B_{2\rho} \times (-\theta \rho^{2s}, 0]$ with $k(t)=\frac{3}{2}\sigma k$ and a cut-off function $0 \leq \xi \leq 1$, $\xi =1$ in $B_\rho$, $\textnormal{supp} \xi \subset B_{3\varrho/2}$, and such that
\begin{align} \label{suxi2}
    |D\xi| \lesssim \frac{1}{\varrho}.
\end{align}
Then, for $w_{\pm}(x,t):= \left (u(x,t)-\frac{3}{2}\sigma k \right )_{\pm}$, dropping the first two terms on the left hand-side, we get
\begin{align} \label{ms:caccioppoli}
&\frac{1}{\texttt{L}}\iint_{Q_{\rho}(\theta)}
w_-(x,t)
\left(
\int_{B_\varrho}
\frac{w_+(y,t)}{|x-y|^{N+2s}}\,\dy
\right)\dx\,\dt
\\ & \qquad \leq
\int_{B_{\frac32\varrho}}
\mathfrak{g}_-\left(u(x,-\theta\varrho^{2s}),\frac{3}{2}\sigma k\right)\,\dx
\notag \\ & \qquad \quad + \gamma \int_{-\theta \rho^{2s}}^{0} \int_{B_{2\varrho}} \int_{B_{2\varrho}} \max \{ w^2_{-}(x,t), w_{-}^2(y,t)\} \frac{|\xi(x)-\xi(y)|^2}{|x-y|^{N+2s}}  \, \dx \, \dy \, \dt 
\nonumber\\
&\qquad\quad
+ 4\texttt{L}  \int_{-\theta \rho^{2s}}^{0} \int_{B_{2\varrho}} \xi^2 w_{-}(x,t) \, \dx  \left ( \operatorname*{ess\,sup}_{x \in B_{\frac{3}{2}\varrho}} \int_{\R^N \setminus B_{2\varrho}} \frac{w_{-}(y,t)}{|x-y|^{N+2s}} \, \dy \right ) \,\dt
\nonumber\\
& \qquad =:I_1+I_2+I_3.
\end{align}
We estimate $I_1,I_2,I_3$. For $I_1$, using  \eqref{suG}, we have
\begin{equation} \label{ii1}
    I_1\leq \gamma(\sigma k)^{q+1} \left |B_{\frac32\varrho} \right |.
\end{equation}
For $I_2$, by \eqref{suxi2} and the definition of $\theta$, we get
\begin{equation} \label{ii2}
    I_2 \leq \frac{\gamma}{\varrho^{2s}}
\int_{-\theta\varrho^{2s}}^0
\int_{B_{2\varrho}}w_-^2\,\dx\,\dt
\leq
\frac{\gamma}{\varrho^{2s}}
(\sigma k)^2
\bigl|Q_{2\varrho}(\theta)\bigr|
\leq
\gamma\theta(\sigma k)^2 |B_\varrho| \leq \gamma(\sigma k)^{q+1}|B_\varrho|.
\end{equation}
For $I_3$ we observe that, by  $w_{-} \leq 3k\sigma/2 + u_{-} $, the estimate
\[
|x-y|
\geq |y|-|x|
\geq \frac{|y|}{4},
\qquad \text{for }
x\in B_{3\rho/2} \text{ and }
y\in\mathbb{R}^N\setminus B_{2\rho},
\]
inequality \eqref{ms2} and the definition of $\theta$, we have
\begin{align}
I_3
&\le
\gamma \sigma k |B_{2\varrho}|
\left(
\delta(\sigma k)^q
+
\operatorname{Tail}(u_-;Q)
\right)
\notag
\\
&\le
\gamma \sigma k |B_{2\varrho}|
\left(
\delta(\sigma k)^q
+
\frac{(\sigma k)^q}{4^q}
\right)
\notag
\\
&\le
\gamma(\sigma k)^{q+1}|B_\varrho|,
\label{ii3}
\end{align}
where we used also the nonnegativity of the solution in $Q$.
Set
\[
A:= \left \{(x,t) \in Q_\rho (\theta):u(x,t)\leq \sigma k \right \}.
\]
For $(x,t)\in A$ we have
\begin{equation} \label{b1}
    w_-(x,t)= \frac{3}{2}\sigma k-u(x,t)
\geq \frac{\sigma k}{2}.
\end{equation}
On the other hand, for $ y \in \{u(\cdot,t)\geq k\}\cap B_\varrho$ we have
\begin{equation} \label{b2}
    w_+(y,t)
> \frac{k}{4},
\end{equation}
since $\sigma<1/2$. Hence, denoting with $L$ the term on the left-hand side of  \eqref{ms:caccioppoli} and using  \eqref{b1},  \eqref{b2} and  \eqref{ms1}, we get
\begin{align} \label{iil}
L & \geq \frac{1}{\texttt{L}}\iint_{Q_{\rho}(\theta)}
w_-(x,t)
\left(
\int_{\{u\geq k \} \cap B_\varrho}
\frac{w_+(y,t)}{|x-y|^{N+2s}}\,\dy 
\right) \dx \, \dt
\nonumber\\
&\geq
\gamma\,\sigma k^2
\int_{A}
\left(
\int_{\{u\geq k\}\cap B_\varrho}
\frac{\dy}{|x-y|^{N+2s}}
\right)\dx \, \dt
\nonumber\\
&\geq
\gamma\,\sigma k^2
\frac{\alpha |B_\varrho|}{\varrho^{N+2s}}
\bigl|A\bigr|
\nonumber\\
&\geq
\frac{\gamma\alpha\sigma k^2}{\varrho^{2s}}
\bigl|A\bigr|.
\end{align}
Therefore, combining  \eqref{ii1}-\eqref{iil} with  \eqref{ms:caccioppoli}, we get
\[ \frac{\gamma\alpha\sigma k^2}{\varrho^{2s}} \left| \left\{u\leq \sigma k \right\}\cap Q_\varrho(\theta)\right | \leq \gamma(\sigma k)^{q+1}|B_\varrho|. \] Since $\theta=\delta(\sigma k)^{q-1}$, this yields
\[ \left| \left\{u\leq \sigma k \right\}\cap Q_\varrho(\theta)\right |\leq \frac{\gamma\sigma}{\delta\alpha} |Q_\varrho(\theta)|
\]
and the conclusion follows.

\end{proof}
\subsection{Proof of expansion of positivity}
We now combine the previous lemmas to prove Theorem \ref{prop:expansion_positivity}. 
\begin{proof}[Proof of Theorem \ref{prop:expansion_positivity}]
 We may assume $(x_0,t_0)=(0,0)$. Let $Q$ be as in the statement. As a restatement, we show that there exist $\delta,\eta\in(0,1)$, depending only on the data and on $\alpha$, such that if
    \begin{equation} \label{c1}
        |\{u(\cdot,0)\geq k\}\cap B_\varrho|\geq \alpha |B_\varrho|
    \end{equation}
and if
\begin{equation} \label{c2}
    \Theta\,\textnormal{Tail}^{\frac{1}{q}}(u_-;{Q})\leq \eta k,
\end{equation}
then
\[
u\geq \eta k
\quad\text{a.e. in }B_{2\varrho}\times
\left(\frac12\delta k^{q-1}\varrho^{2s},\delta k^{q-1}\varrho^{2s}\right].
\]
Let $\Theta$ be as in~\eqref{def:Theta}. At the initial time $t_0=0$, we regard the measure assumption in the larger ball $B_{4\varrho}$ and replace $\alpha$ by $4^{-N}\alpha$. Take in  \eqref{c2} $\eta \leq \eta_0$, where $\eta_0$ is detected in Lemma \ref{lem3}. Then, by  \eqref{c1} and  \eqref{c2}, we can apply Lemma \ref{lem3} and obtain
 $\delta, \varepsilon \in (0,1)$ depending on data and $\alpha$, such that
\[
|\{u(\cdot,t)\geq \varepsilon k\}\cap B_{4\varrho}|
\geq
\frac{\alpha}{2}\,4^{-N}|B_{4\varrho}|
\qquad
\text{for all }t\in\bigl(0,\delta k^{q-1}(4\varrho)^{2s}\bigr].
\]
Therefore, we can apply Lemma \ref{lem4} in the cylinder 
$$Q_{4\rho}\left (\frac{\delta}{2}(\sigma\varepsilon k)^{q-1} \right ):=B_{4\varrho}\times \left(\bar t-\frac{\delta}{2}(\sigma\varepsilon k)^{q-1}(4\varrho)^{2s},\bar t\right), \qquad \text{with } \bar t\in\left(\frac{\delta}{2}(\sigma\varepsilon k)^{q-1}(4\varrho)^{2s}, \delta k^{q-1}(4\varrho)^{2s}\right],$$
and with $k$ replaced by $\varepsilon k$ and $\alpha$ replaced by $\frac\alpha2 4^{-N}$. This is admissible since $\sigma,\varepsilon\in(0,1)$ and $q>1$, so that
\(
\delta(\sigma\varepsilon k)^{q-1}/2\leq \delta k^{q-1}.
\)
Therefore the backward cylinders considered above are contained in
\(
B_{4\varrho}\times\bigl(0,\delta k^{q-1}(4\varrho)^{2s}\bigr].
\)
So, enforcing
$$ \eta := \min \left \{\eta_0, \frac{\sigma \varepsilon}{4} \right \},$$
we obtain
\[ \left|\{u\leq \sigma\varepsilon k\}\cap Q_{4\rho}\left (\frac{\delta}{2}(\sigma\varepsilon k)^{q-1} \right )\right | \leq \frac{\gamma\sigma}{\delta\alpha}\left|Q_{4\rho}\left (\frac{\delta}{2}(\sigma\varepsilon k)^{q-1} \right ) \right|. \]
Now, let $\nu\in(0,1)$ be the constant determined in Lemma \ref{lem1}, depending only on the data and $\delta$. Recalling that the constant $\gamma$ in Lemma \ref{lem4} is independent of $\sigma$, we choose $\sigma\in(0,1/2)$ small enough, according to Lemma \ref{lem4}, so that
\[
\frac{\gamma\sigma}{\delta\alpha}<\nu .
\] 
Therefore, 
we can apply Lemma \ref{lem1} in the cylinders 
\[ B_{4\varrho}\times \left(\bar t-\frac{\delta}{2}(\sigma\varepsilon k)^{q-1}(4\varrho)^{2s},\bar t\right), \qquad \text{for every }\bar t\in \left(\frac{\delta}{2}(\sigma\varepsilon k)^{q-1}(4\varrho)^{2s}, \delta k^{q-1}(4\varrho)^{2s}\right], \]
and with $k$ replaced by $\sigma\varepsilon k$. Since $\bar{t}$ is arbitrary, we conclude \[
u\geq \frac{\sigma\varepsilon k}{4}
\quad\text{almost everywhere in }\;
B_{2\varrho}\times
\left(\frac{\delta}{2}(\sigma\varepsilon k)^{q-1}(4\varrho)^{2s}, \delta k^{q-1}(4\varrho)^{2s}\right).
\]
Choosing $\sigma \in (0,1/2)$ such that $4^{2s}(\sigma \varepsilon)^{q-1} \leq 1$, the conclusion follows.

\end{proof}
\noindent

\section{Appendix} \label{sec5}\label{timemollification}
\noindent
In this section we recall a time mollification argument which justifies the use of test functions depending on the solution itself and yields Lemma \ref{cacc} and the inequality \eqref{tst}.

For all $v \in L^1(Q_{R,S}(x_0,t_0))$ and all $h\in(0,1)$, we define two finite convolutions of $v$ with exponential weight functions by setting, for all $(x,t)\in\Omega_T$,
\begin{equation} \label{vh}
v_h(x,t) = \int_{t_0-S}^t\frac{ v(x,\tau)}{h}e^\frac{\tau-t}{h}\,\d\tau,
\qquad
v_{\bar h}(x,t) = \int_t^{t_0} \frac{v(x,\tau)}{h}e^\frac{t-\tau}{h}\,\d\tau.
\end{equation}
\noindent
In the following proposition we give some properties of these mollifications, the proof can be found in \cite{CCI}.
\begin{proposition}
\label{con}
    Let $\Omega\subset \R^N$ an open set, $v \in L^r(\Omega_T)$ for some $r \in [1,\infty)$. Then for all $h \in (0,1)$ we have:
    \begin{enumroman}
    \item\label{con1} $v_h,v_{\bar h}$ are differentiable in $t$ with
\[\frac{\partial v_h}{\partial t}(x,t) = \frac{v(x,t)-v_h(x,t)}{h}, \ \frac{\partial v_{\bar h}(x,t)}{\partial t} = \frac{v_{\bar h}(x,t)-v(x,t)}{h};\]
    \item\label{con2} $v_h,v_{\bar h}\in L^r(\Omega_T)$, $||v_h||_{L^r(\Omega_T)} \leq||v||_{L^r(\Omega_T)}$ and $v_h,v_{\bar h}\to v$ in $L^r(\Omega_T)$ as $h\to 0^+$, also if $v\ge 0$ in $\Omega_T$ then $0\leq v_h,v_{\bar h}$ in $\Omega_T$;
        \item \label{con3} for $r>1$ and for all $\phi \in L^{r'}(\Omega_T)$ we have
        \begin{equation*}
            \iint_{B_R\times[0,T]} v(x,t) \partial_t\phi_{\bar{h}}(x,t) \, \dx \, \dt = - \iint_{B_R \times[0,T]} \partial_t[v]_h(x,t)\phi(x,t) \, \dx \, \dt
        \end{equation*}
        \item\label{con4} if $v\in L^r(0,T,W^{s,r}(\Omega'))$ for some $\Omega'\Subset\Omega$, then $v_h,v_{\bar h}\in L^r(0,T,W^{s,r}(\Omega'))$ with
\[\int_0^T\iint_{\Omega'\times\Omega'}\frac{|v_h(x,\tau)-v_h(y,\tau)|^r}{|x-y|^{N+rs}}\,\dx\,\dy\,\d\tau \leq\|v\|_{L^r(0,T,W^{s,r}(\Omega'))}^r.\]
    \end{enumroman}
\end{proposition}

\noindent
In order to prove Lemma \ref{cacc}, we need the following result that can be proved as \cite[Lemma B.2]{MY}.
\begin{lemma}
\label{weakformlemma}
    Let $B_R \subset \Omega$, $u$ be a local, nonnegative, weak sub(super)-solution of \eqref{eq:PDE}, $\phi \in L^2(0,T; W^{s,2}(B_R)) \cap L^{q+1}(0,T;L^{q+1}(B_R))$ such that $\phi \geq 0$ in $B_R \times [0,T]$ and $\phi \equiv 0$ a.e. in $B_R \backslash B_\rho \times [0,T]$ for some $\rho<R$. Then it holds for all $(t_1,t_2) \subset [0,T]$
    \begin{equation}
    \label{weakform}
        \begin{aligned}
            & \int_{B_R} (u^q(x,t) - (u^q(x,t))_h)\phi_{\bar{h}} \, \dx \biggl |_{t=t_1}^{t=t_2} + \iint_{B_R \times (t_1,t_2)} \partial_t(u^q(x,t))_h \phi(x,t) \, \dx \, \dt \\
            & \qquad + \int_{t_1}^{t_2} \iint_{\R^N \times \R^N} (u(x,t)-u(y,t))(\phi_{\bar{h}}(x,t)-\phi_{\bar{h}}(y,t)) K(x,y,t) \, \dx \, \dy \, \dt \leq (\ge) 0 \, .
        \end{aligned}
    \end{equation}
\end{lemma}
\noindent
We are ready to prove the Caccioppoli inequalities of Lemma \ref{cacc}.
\begin{proof}[Proof of Lemma \ref{cacc}]
   Let $u$, $k$, $\zeta$, $\xi$ and $w_{\pm}$ be as in Lemma \ref{cacc}.   For $\tau_0 \in (t_0-S,t_0]$, $\epsilon > 0$ small enough we define a time-dependent cut-off functions by
\[
\psi^\varepsilon(t):=
\begin{cases}
0,
& t\in[0,t_0-S],\\[2mm]
\dfrac{t-(t_0-S)}{\varepsilon},
& t\in(t_0-S,t_0-S+\varepsilon],\\[3mm]
1,
& t\in(t_0-S+\varepsilon,\tau_0-\varepsilon],\\[2mm]
\dfrac{\tau_0-t}{\varepsilon},
& t\in(\tau_0-\varepsilon,\tau_0],\\[3mm]
0,
& t\in(\tau_0,T],
\end{cases}
\]
    We notice $\psi^{\eps}$ is a Lipschitz mapping. We treat first the case of sub-solution. Now, we consider the following admissible test function for \eqref{weakform}
    \begin{align*}
        \phi^{\eps}(x,t) := \begin{cases}
            \zeta^2(x,t) \psi^{\varepsilon}(t) w(x,t)_{+}, \quad &  t \in [t_0-S,\tau_0] \\ 0 \quad &  t \in (\tau_0,T]
        \end{cases} ,
    \end{align*}
    therefore by Lemma \ref{weakformlemma}
    \begin{equation} \label{nnclm}
        \begin{aligned}
          0 \geq & \int_{B_R} (u^q(x,t)-(u^q(x,t))_h) \phi_{\bar{h}}^{\eps} \, \dx \biggl|_{t=t_0-S}^{\tau_0}  + \int_{t_0-S}^{\tau_0} \int_{B_R} \partial_t(u^q(x,t))_h\phi^{\eps} \, \dx \, \dt \\
           & \qquad + \int_{t_0-S}^{\tau_0} \iint_{\R^N \times \R^N}(u(x,t)-u(y,t))(\phi_{\bar{h}}^{\eps}(x,t)-\phi_{\bar{h}}^{\eps}(y,t))K(x,y,t) \, \dx \, \dy \, \dt \\
         & =:   R+L+D \,,
        \end{aligned}
    \end{equation}
    where we are using $t_1=t_0-S$ and $t_2=\tau_0$. Let us start analyzing the term $R$. Observe that
    \begin{equation*}
        \begin{aligned}
            R=& \int_{B_R} (u^q(x,\tau_0) - (u^q(x,\tau_0))_h)\int_{\tau_0}^T \frac{e^{\frac{\tau_0-\tau}{h}}}{h}w_{+}(x,\tau)\zeta^2(x,\tau) \psi^{\eps}(\tau)    \, \d\tau \, \dx \\
            & \qquad - \int_{B_R} (u^q(x,t_0-S) - (u^q(x,t_0-S))_h)\int_{t_0-S}^T \frac{e^{\frac{t_0-S-\tau}{h}}}{h}w_{+}(x,\tau)\zeta^2(x,\tau) \psi^{\eps}(\tau) \, \d\tau \, \dx
        \end{aligned}
    \end{equation*}
    Now, the first term is always zero by the definition of $\psi_\varepsilon$, while the second goes to $0$ when $h \to 0$ for every $\epsilon >0$. Then, by H\"older inequality and the dominated convergence theorem
    \begin{equation} \label{forR}
        \lim_{h \to 0} R = \lim_{h \to 0} R_{\varepsilon,h} =0, \qquad \text{for every } \varepsilon >0.
    \end{equation}
Now, let us focus on the parabolic term $L$. We decompose the term into the two parts
    \begin{align*}
         L&=\iint_{B_R \times [0,T]} \partial_t (u^{q})_h (x,t) \zeta^2(x,t)  w_{+}(x,t) \psi^{\eps}(t) \,\dx\,\dt \\
        &= \iint_{B_R \times [0,T]}  \partial_t (u^{q})_h (x,t) \zeta^2(x,t) \psi^{\eps}(t) \big[ w_{+}(x,t) - \big( ( u^{q} )_{h}^{\frac{1}{q}}(x,t) - k(t) \big)_{+} \big] \,\dx\,\dt \\
        &\phantom{=}
        \qquad + \iint_{B_R \times [0,T]} \partial_t ( u^{q} )_{h} (x,t) \zeta^2(x,t) \psi^{\eps}(t) \left( ( u^{q} )_{h}^{\frac{1}{q}}(x,t) - k(t) \right)_{+}  \,\dx\,\dt =: \mathrm{L}_{1,1} + \mathrm{L}_{1,2}.
    \end{align*}

    \noindent
    Now, term $\mathrm{L}_{1,1}$ is nonnegative and can be omitted. In fact, defining $\tilde{g}(\tau)=(\tau^{\frac{1}{q}}-k)_{+}$, we observe that this is non decreasing and so, using Proposition \ref{con} \ref{con1}, we have
    \begin{equation*}
        L_{1,1}=\iint_{B_R \times [0,T]}  \frac{u^q(x,t)-(u^q)_h(x,t)}{h}  \zeta^2(x,t) \psi^{\eps}(t) \big[ \tilde{g}(u^q) - \tilde{g}((u^q)_h) \big] \,\dx\,\dt \geq 0 \, .
    \end{equation*}
    
    \noindent
    We proceed with $\mathrm{L}_{1,2}$. By the time dependence of $k(\cdot)$ and definition \eqref{GG}, it follows that
    \begin{equation*}
        \begin{aligned}
            \partial_t \mathfrak{g}_{+}((u^q)_h^{\frac{1}{q}},k(t)) =& \partial_t q \int_{k(t)}^{(u^q)_h^{\frac{1}{q}}} |\theta|^{q-1} (\theta-k(t))_{+} \, \d\theta  \\
            =& q (u^q)_h^{\frac{q-1}{q}}\left((u^q)_h^{\frac{1}{q}}-k(t) \right )_{+} \frac{1}{q} (u^q)_h^{\frac{1-q}{q}} \partial_t((u^q)_h)  - k'(t) q\int_{k(t)}^{(u^q)_h^{\frac{1}{q}}} |\theta|^{q-1} \, \d\theta \\
            =& \left((u^q)_h^{\frac{1}{q}}-k(t) \right)_{+} \partial_t((u^q)_h) - k'(t) \big((u^q)_h-k^q(t) \big )_{+} \, .
        \end{aligned}
    \end{equation*}
    Then, 
    \begin{align*}
        \mathrm{L}_{1,2} &=\iint_{B_R \times [0,T]}  \zeta^2(x,t)\psi^{\eps}(t) \partial_t \mathfrak{g}_{+}\left(( u^{q} )_{h}^{\frac{1}{q}}(x,t),k(t)\right)   \,\dx\,\dt \\
        &\phantom{=}
        + \iint_{B_R \times [0,T]} \zeta^2(x,t) \psi^{\eps}(t) k^{\prime}(t) \big( ( u^{q} )_{h}(x,t) - k^{q}(t) \big)_{+}  \,\dx\,\dt =: \mathrm{L}_{1,2}^{\prime} + \mathrm{L}_{1,2}^{\prime \prime}.
    \end{align*}
    Since $k(\cdot)$ is absolutely continuous in $[t_0-S, t_0]$, the term $\mathrm{L}_{1,2}^{\prime }$ can be treated similarly as in~\cite[equation B.13]{MY}. Therefore, we deduce
    \begin{align}\label{mol_time_part_energy_ineq}
        \lim_{\varepsilon \to 0}\lim_{h \to 0} \mathrm{L}_{1,2}^{\prime} &\geq \int_{B_{R}} \zeta^2(x,t) \mathfrak{g}_{+}( u(x,t), k(t)) \,\dx \bigg|_{t_0 - S}^{\tau_0} - \iint_{Q_{R,S}} \big| \partial_t \zeta^2 \big| \mathfrak{g}_{+}( u(x,t), k(t)) \,\dx\,\dt \nonumber \\
        &= \int_{B_{R}} \zeta^2(x,\tau_0) \mathfrak{g}_{+}( u(x,\tau_0), k(\tau_0)) \,\dx  - \iint_{Q_{R,S}} \big| \partial_t \zeta^2 \big| \mathfrak{g}_{+}( u(x,t), k(t)) \,\dx\,\dt,
    \end{align}
   where we used the support of $\zeta$. As regards $\mathrm{L}_{1,2}''$, apply Proposition \ref{con} \ref{con2} obtaining $(u^{q})_{h} \to u^{q}$ in $L^{1}(Q_{R,S})$ as $h \to 0$ and $\Vert (u^{q})_{h} \Vert_{L^{1}(Q_{R,S})} \leq \Vert u^{q} \Vert_{L^{1}(Q_{R,S})}$. Now, send  $\eps \to 0$ and obtain $\psi_\varepsilon \to 1$ on its support. Hence, by the same argument as  \cite[Lemma 2.6]{FDF} and by the absolute continuity of $k$, we get
    \begin{align*}
        \lim_{\varepsilon \to 0} \lim_{h \to 0} \mathrm{L}_{1,2}^{\prime \prime} = \int_{t_0 - S}^{\tau_0} \int_{B_R} \zeta^2 (x,t) k^{\prime}(t) \big( (u^{q} (x,t) - k^{q}(t) \big)_{+}  \,\dx\,\dt.
    \end{align*}
    Moreover, using $k'\ge 0$, for any $(x,t) \in Q_{R,S}$, it holds
\begin{align}\label{mol_ineq:k_deriv_term_lower_bound}
        k^{\prime}(t) \big( (u^{q} (x,t) - k^{q}(t) \big)_{+}  \geq q k^{\prime}(t) k^{q - 1}(t) \big( (u (x,t) - k(t) \big)_{+} = ( k^{q}(t) )^{\prime} \big( (u (x,t) - k(t) \big)_{+}.
    \end{align}
Therefore, we get \begin{align}\label{mol_ineq_lower_bound_for_time_deriv_term}
        \liminf_{\varepsilon \to 0} \liminf_{h \to 0} \mathrm{L} &\geq \int_{B_{R}} \zeta^2(x,\tau_0) \mathfrak{g}_{+}( u(x,\tau_0), k(\tau_0)) \,\dx  - \iint_{Q_{R,S}} \big| \partial_t \zeta^2 \big| \mathfrak{g}_{+}( u(x,t), k(t)) \,\dx\,\dt \nonumber \\
        &\phantom{=}
        + \int_{t_0 - S}^{\tau_0} \int_{B_R} (k^{q}(t))^{\prime} \zeta^{2}  w_{+}(x,t) \,\dx\,\dt \, .
    \end{align}
    As regards the diffusive term $D$, it can be treated similarly to term $H_2$ in Lemma \ref{mol} below. Therefore, we deduce
\begin{equation*}
    \lim_{\varepsilon \to 0} \lim_{h \to 0} D = \int_{t_0-S}^{\tau_0} \iint_{\R^N \times \R^N}(u(x,t)-u(y,t))(\phi(x,t)-\phi(y,t))K(x,y,t) \, \dx \, \dy \, \dt,
\end{equation*}
where 
\begin{equation} \label{mollified1}
    \phi(x,t):= \zeta^2(x,t) w_{+}(x,t).
\end{equation}
    \noindent
   Now, due to $w_{+} w_{-} = 0$, we obtain that
    \begin{align*}
        (u(y,t)-u(x,t)) &\big[ \zeta^2 (y,t) w_{+}(y,t) - \zeta^2 (x,t) w_{+}(x,t) \big] \\
        &= \big[ \zeta (y,t) w_{+}(y,t) - \zeta (x,t) w_{+}(x,t)  \big]^2 +\zeta^2 (x,t) w_{-}(y,t) w_{+}(x,t) \\
        &\phantom{=}
        + \zeta^2 (y,t) w_{-}(x,t) w_{+}(y,t) - w_{+}(y,t) w_{+}(x,t) \big[ \zeta(x,t) - \zeta (y,t) \big]^2 .
    \end{align*}
    Therefore, we get 
\begin{align}\label{ineq:diffusion_term_decomp_EElem}
      D &=  \int_{t_0 - S}^{\tau_0}  \iint_{\mathbb{R}^{N} \times \mathbb{R}^{N}}  (u(x,t) - u(y,t)) K(x,y,t)  ( \zeta^{2}(x,t) w_{+}(x,t) - \zeta^{2}(y,t) w_{+}(y,t) ) \,\dx\,\dy\,\dt \nonumber \\
        &\geq \int_{t_0 - S}^{\tau_0} \iint_{\mathbb{R}^{N} \times \mathbb{R}^{N}} \big[ \zeta (y,t) w_{+}(y,t) - \zeta (x,t) w_{+}(x,t)  \big]^2 K(x,y,t) \,\dx\,\dy\,\dt  \nonumber \\
        &\phantom{=}
        \quad+ \int_{t_0 - S}^{\tau_0} \iint_{B_R \times \mathbb{R}^{N}  } \zeta^2 (x,t) w_{-}(y,t) w_{+}(x,t) K(x,y,t) \,\dx\,\dy\,\dt \nonumber \\
        &\phantom{=}
        \quad- \int_{t_0 - S}^{\tau_0} \iint_{\mathbb{R}^{N} \times \mathbb{R}^{N}} w_{+}(y,t) w_{+}(x,t) \big[ \zeta(x,t) - \zeta (y,t) \big]^2 K(x,y,t) \,\dx\,\dy\,\dt. 
    \end{align}
    The last term in~\eqref{ineq:diffusion_term_decomp_EElem} can be expressed as
    \begin{align} \label{d1}
        \int_{t_0 - S}^{\tau_0} \iint_{\mathbb{R}^{N} \times \mathbb{R}^{N}} & w_{+}(y,t) w_{+}(x,t) \big[ \zeta(x,t) - \zeta (y,t) \big]^2 K(x,y,t) \,\dx\,\dy\,\dt \notag \\
        &= \int_{t_0 - S}^{\tau_0} \iint_{B_R \times B_R } w_{+}(y,t) w_{+}(x,t) \big[ \zeta(x,t) - \zeta (y,t) \big]^2 K(x,y,t) \,\dx\,\dy\,\dt \notag \\
        &\phantom{=}
       \quad  + 2  \int_{t_0 - S}^{\tau_0} \iint_{B_R \times ( \mathbb{R}^{N} \setminus B_R ) } \zeta^2  w_{+}(x,t) w_{+}(y,t) K(x,y,t) \,\dx\,\dy\,\dt.
    \end{align}
    Note that the first term on the right-hand side in the above display can be further estimated by using
    \begin{align} \label{d2}
         w_{+}(y) w_{+}(x) \leq \max\{ w_{+}^{2}(x,t), w_{+}^{2}(y,t) \}.
    \end{align}

    \noindent
    Then, using \eqref{nnclm}, \eqref{forR}, \eqref{mol_ineq_lower_bound_for_time_deriv_term}, \eqref{d1}, \eqref{d2}, \eqref{k}, we get
    \begin{align}\label{prelim_energy_estimate_1}
        &\int_{B_R(x_0)}\zeta^2(x,\tau_0)\mathfrak{g}_+(u(x,\tau_0),k(\tau_0))\,\dx
    \notag\\
    &\quad+\frac{1}{\texttt{L}}\int_{t_0-S}^{\tau_0}\iint_{B_R(x_0)\times B_R(x_0)}
    \frac{|\zeta w_+(x,t)- \zeta w_+(y,t)|^2}{|x-y|^{N+2s}}\,\dx\,\dy\,\dt
    \notag\\
    &\quad+\frac{1}{\texttt{L}}\int_{t_0-S}^{\tau_0}\int_{B_R(x_0)}\zeta^2(x,t)w_{+}(x,t)
    \left(\int_{\mathbb R^N}\frac{w_{-}(y,t)}{|x-y|^{N+2s}}\,\dy\right)\dx\,\dt
    \notag\\
    &\leq  \gamma\int_{t_0-S}^{t_0}\iint_{B_R(x_0)\times B_R(x_0)}
    \max\{w_{+}^2(x,t),w_{+}^2(y,t)\}
    \frac{|\zeta(x,t)-\zeta(y,t)|^2}{|x-y|^{N+2s}}\,\dx\,\dy\,\dt
    \notag\\
    &\quad+2\textnormal{\texttt{L}}\int_{t_0-S}^{t_0}\int_{B_R(x_0)}\zeta^2(x,t)w_{+}(x,t) \left(\int_{\mathbb R^N\setminus B_R(x_0)}\frac{w_{+}(y,t)}{|x-y|^{N+2s}}\,\dy\right)\dx\,\dt 
    \notag\\
    &\quad-\int_{t_0-S}^{\tau_0}\int_{B_R(x_0)}(k^q(t))'\zeta^2(x,t)w_{+}(x,t)\,\dx\,\dt + \gamma\int_{t_0-S}^{t_0}\int_{B_R(x_0)}|\partial_t \zeta^2 (x,t)|\mathfrak{g}_{+}(u(x,t),k(t))\,\dx\,\dt.
    \end{align}
    Now, we argue similar as in~\cite[Proof of Proposition 2.7]{FDF}.
    In the estimate above, we first choose $\tau_0 = t_0$ such that the second and third term on the left-hand side be bounded by those on the right-hand side of~\eqref{prelim_energy_estimate_1}. Next, $\tau_0 \in (t_0- S, t_0)$ is arbitrary, dropping the second and third nonnegative term on the left-hand side of~\eqref{prelim_energy_estimate_1}, while
     taking the supremum over $\tau_0 \in (t_0- S, t_0)$, we deduce that the first integral on the left-hand side of~\eqref{prelim_energy_estimate_1} is also bounded by the right-hand side. Thus, \eqref{cacc1} is proved for sub-solution.
     \vskip 0.1 cm \noindent
         With similar calculations, it is possible to prove \eqref{cacc1} for a super-solution with $w_{-}$, $\mathfrak{g}_{-}$ in place of $w_{+}$, $\mathfrak{g}_{+}$. At this point, we want to highlight only the difference. First of all note that for super-solutions we obtain for the analog of $\mathrm{L}_{1,2}$ that
     \begin{align*}
         \mathrm{L}_{1,2} :&= - \iint_{Q_{R,S}} \partial_t ( u^{q} )_{h} (x,t) \zeta^2(x,t) \left( ( u^{q} )_{h}^{\frac{1}{q}}(x,t) - k(t) \right)_{-}  \,\dx\,\dt \\
         &= \iint_{Q_{R,S}}  \zeta^2(x,t) \partial_t \mathfrak{g}_{-}(( u^{q} )_{h}^{\frac{1}{q}}(x,t),k(t))   \,\dx\,\dt \\
        &\phantom{=}
        - \iint_{Q_{R,S}} \zeta^2(x,t) k^{\prime}(t) \big( ( u^{q} )_{h}(x,t) - k^{q}(t) \big)_{-}  \,\dx\,\dt.
     \end{align*}
     Most of the calculations can be performed analogously. However, instead of~\eqref{mol_ineq:k_deriv_term_lower_bound} we deduce, using $k'\le0$,
     \begin{align*}
         k^{\prime}(t) \big( (u^{q} (x,t) - k^{q}(t) \big)_{-} \leq  k^{\prime}(t) k^{q-1}(t) w_{-}(x,t) = \frac{1}{q} ( k^{q}(t) )^{\prime} \big( (u (x,t) - k(t) \big)_{-}.
     \end{align*}
     Thus, obtaining \eqref{cacc1} for super-solution.
         \vskip 0.1 cm \noindent 
         For the Caccioppoli inequality \eqref{cacc2}, we consider
             \begin{align} \label{uu}
        \varphi^{\eps}(x,t) := \begin{cases}
    \xi^2(x)\psi_\varepsilon(t)w_+(x,t), \quad &  t \in [t_0-S,\tau_0] \\ 0 \quad &  t \in (\tau_0,T]
        \end{cases} ,
    \end{align}
in the sub-solution case. We argue as in the proof of \eqref{cacc1}. Let us only highlight the main differences compared to the calculations from above. Call $R', L', D'$ the new terms in \eqref{nnclm}, obtained by testing \eqref{weakform} with \eqref{uu}. Now, terms $R', D'$ are treated as $R, D$ above. As regards $L'$, we observe that, due to the time-independence of $\xi$ the second term in~\eqref{mol_ineq_lower_bound_for_time_deriv_term} vanishes, and we deduce
    \begin{align*}
        \liminf_{\varepsilon \to 0} \liminf_{h \to 0} L' \geq & \int_{B_{R}} \xi^2(x) \mathfrak{g}_{+}( u(x,\tau_0), k(\tau_0)) \,\dx  - \int_{B_{R}} \xi^2(x)\mathfrak{g}_{+}( u(x,t_0 - S), k(t_0 - S)) \,\dx \\
        &\phantom{=}
        + \int_{t_0 - S}^{\tau_0} \int_{B_R} \xi^2(x) (k^{q}(t))^{\prime}   w_{+}(x,t) \,\dx\,\dt.
    \end{align*}
 Combining the new estimates, \eqref{cacc2} follows. The  case of super-solution can be treated as above. The proof is ended.
\end{proof}

\noindent
The following lemma provides a rigorous justification of inequality \eqref{tst1}. For simplicity it is stated for $(x_0,\bar{t})=(0,0)$.

\begin{lemma}\label{mol}
Let $q>1$ and let $u$ be a local, nonnegative, weak super-solution to~\eqref{eq:PDE} in $\Omega_T$. Consider any cylinder $B_{ \rho}(0) \times (0,\tau) \Subset \Omega_{T}$ and let $\hat B\subset B$, $\xi$, $\nu$ be as in \eqref{BB} and \eqref{property_of_tst_fct}, \eqref{lambda}
, respectively. Then there exists a constant $\gamma>0$, depending only on the data, such that
\begin{align*}
0\leq{}&
\tau^{\frac12}
\int_{\hat B}\xi^2(x)
\left[
\int_0^{u(x,\tau)}
\theta^{q-1}(\theta+\nu)^{\frac{1-q}{2}}\,\d\theta
\right]\,\dx
-\frac12
\int_0^\tau t^{-\frac12}
\int_{\hat B}\xi^2(x)
\left[
\int_0^{u(x,t)}
\theta^{q-1}(\theta+\nu)^{\frac{1-q}{2}}\,\d\theta
\right]\,\dx\,\dt
\\
&+\gamma
\int_0^\tau t^{\frac12}
\iint_{\R^N\times\R^N}
\bigl(u(x,t)-u(y,t)\bigr)
\left[
\bigl(u(x,t)+\nu\bigr)^{\frac{1-q}{2}}\xi^2(x)
-
\bigl(u(y,t)+\nu\bigr)^{\frac{1-q}{2}}\xi^2(y)
\right]
K(x,y,t)\,\dx\,\dy\,\dt.
\end{align*}
\end{lemma}
\begin{proof}
Fix $T_* > \tau$ such that $B \times (0,T_*) \Subset \Omega_T$ and define $u_{\bar{h}}$ in \eqref{vh} with $t_0 = T_*$. For $\varepsilon \in (0,\tau/2)$, define the Lipschitz function
$\psi_{\varepsilon}:[0,T]\to[0,1]$ by
\[
\psi_{\varepsilon}(t)
:=
\begin{cases}
\dfrac{t}{\varepsilon},
& 0 \leq t < \varepsilon,\\[2mm]
1,
& \varepsilon \leq t < \tau -\varepsilon,\\[2mm]
\dfrac{\tau-t}{\varepsilon},
& \tau-\varepsilon \leq t < \tau ,\\[2mm]
0,
& \tau \leq t \leq T.
\end{cases}
\]
Now, fix $h \in (0,1)$ and set, for all $(x,t) \in \R^N \times (0,T)$,
\begin{equation*}
    \varphi_h^{\eps} (x,t)
    := \begin{cases}
    t^{\frac{1}{2}}
    (u_{\bar{h}} (x,t) + \nu)^{\frac{1-q}{2}}
    \xi^2(x) \psi_{\eps}(t), \quad & \text{in }  B \times (0,T_*)\\ 0, \quad & \text{otherwise}.
    \end{cases}
\end{equation*}
By \cite[Lemma B.1]{MY}, the function
\[
\varphi^\eps_h
\in
W^{1,q+1}_{\rm loc}(0,T;L^{q+1}(\hat B))
\cap
L^2_{\rm loc}(0,T;W^{s,2}_0(\hat B))
\]
is a suitable nonnegative test function for~\eqref{eq}. Moreover,
$\varphi^\eps_h(\cdot,t)=0$ for $t \notin(0,\tau)$ and
$\xi=0$ in $\hat B^c$. Therefore, using \eqref{eq}, we obtain
\begin{align}\label{mol1}
0 \leq {}&
-\int_{0}^{{T_*}}\int_{\hat B}
u^q(x,t)
{\partial_t \varphi^\eps_h}(x,t)
\,\dx\,\dt
\nonumber\\
&+
\int_{0}^{{T_*}}
\iint_{\R^N\times\R^N}
(u(x,t)-u(y,t))
\bigl(
\varphi^\eps_h(x,t)-\varphi^\eps_h(y,t)
\bigr)
K(x,y,t)
\,\dx\,\dy\,\dt
\nonumber\\
=:{}& H_1+H_2.
\end{align}
We now estimate both $H_1$ and $H_2$ from above, and then pass to the limit as $h,\eps \to 0^{+}$.

We first focus on the parabolic term $H_1$. Adding and subtracting the term
$u_{\bar h}^{q}\partial_t\varphi_h^\eps$, using Proposition \ref{con} \ref{con1},
the chain rule and the definition of $\psi_\eps$, we obtain
\begin{align*}
H_1
={}&-\int_0^{T_*}\int_{\hat B}u^q(x,t)\partial_t\varphi_h^\eps(x,t)\,\dx\,\dt
\\
={}&-\int_0^{T_*}\int_{\hat B}u_{\bar h}^q(x,t)\partial_t\varphi_h^\eps(x,t)\,\dx\,\dt
+\int_0^{T_*}\int_{\hat B}\bigl(u_{\bar h}^q(x,t)-u^q(x,t)\bigr)
\partial_t\varphi_h^\eps(x,t)\,\dx\,\dt
\\
={}&\int_0^{T_*}\int_{\hat B}\partial_t u_{\bar h}^q(x,t)\varphi_h^\eps(x,t)\,\dx\,\dt
\\
&+\int_0^{T_*}\int_{\hat B}\bigl(u_{\bar h}^q(x,t)-u^q(x,t)\bigr)
\bigl(u_{\bar h}(x,t)+\nu\bigr)^{\frac{1-q}{2}}\xi^2(x)
\partial_t\bigl(t^{\frac12}\psi_\eps(t)\bigr)\,\dx\,\dt
\\
&+\frac{1-q}{2}\int_0^{T_*}\int_{\hat B}\bigl(u_{\bar h}^q(x,t)-u^q(x,t)\bigr)
\bigl(u_{\bar h}(x,t)+\nu\bigr)^{-\frac{q+1}{2}}
\frac{u_{\bar h}(x,t)-u(x,t)}{h}
\xi^2(x)t^{\frac12}\psi_\eps(t)\,\dx\,\dt
\\
=:{}&H_3+H_4+H_5.
\end{align*}
\noindent
Observe that $H_5 \leq 0$ since the map $u \mapsto u^q$ is monotone. The term $H_4$ goes to $0$ as $h \to 0^{+}$ by Proposition \ref{con} \ref{con2} with $r=q$.

As regards $H_3$, using integration by parts and the fact that
$\psi_\eps(0)=\psi_\eps({T_*})=0$, we obtain
\begin{align*}
H_3
={}&\int_0^{T_*}\int_{\hat B}\partial_t u_{\bar h}^q(x,t)
\bigl(u_{\bar h}(x,t)+\nu\bigr)^{\frac{1-q}{2}}
\xi^2(x) t^{\frac12}\psi_\eps(t)\,\dx\,\dt
\\
={}&q\int_0^{T_*}\int_{\hat B}\xi^2(x)t^{\frac12}\psi_\eps(t)
\partial_t\left[
\int_0^{u_{\bar h}(x,t)}
(\theta+\nu)^{\frac{1-q}{2}}\theta^{q-1}\,d\theta
\right]\,\dx\,\dt
\\
={}&-q\int_0^{T_*}\int_{\hat B}\xi^2(x)t^{\frac12}
\partial_t\psi_\eps(t)
\left[
\int_0^{u_{\bar h}(x,t)}
(\theta+\nu)^{\frac{1-q}{2}}\theta^{q-1}\,d\theta
\right]\,\dx\,\dt
\\
&-\frac{q}{2}\int_0^{T_*}\int_{\hat B}\xi^2(x)t^{-\frac12}
\psi_\eps(t)
\left[
\int_0^{u_{\bar h}(x,t)}
(\theta+\nu)^{\frac{1-q}{2}}\theta^{q-1}\,d\theta
\right]\,\dx\,\dt
\\
\longrightarrow{}&
q \tau^{\frac12}\int_{\hat B}\xi^2(x)
\left[
\int_0^{u(x,\tau)}
(\theta+\nu)^{\frac{1-q}{2}}\theta^{q-1}\,d\theta
\right]\,\dx
\\
&-\frac{q}{2}\int_0^\tau\int_{\hat B}\xi^2(x)t^{-\frac12}
\left[
\int_0^{u(x,t)}
(\theta+\nu)^{\frac{1-q}{2}}\theta^{q-1}\,d\theta
\right]\,\dx\,\dt,
\end{align*}
where we first let $h\to0^+$, using Proposition \ref{con} \ref{con2} with $r=q+1$, and then $\eps\to0^+$. Therefore, we estimate $H_1$ in \eqref{mol1} by
\begin{equation}\label{mol2}
\begin{aligned}
\limsup_{\eps\to0^+}\limsup_{h\to0^+}\frac{H_1}{q}
\leq{}&
\tau^{\frac12}\int_{\hat B}\xi^2(x)
\left[\int_0^{u(x,\tau)}(\theta+\nu)^{\frac{1-q}{2}}\theta^{q-1}\,d\theta\right]\,\dx
\\
&-\frac{1}{2}\int_0^\tau\int_{\hat B}\xi^2(x)t^{-\frac12}
\left[\int_0^{u(x,t)}(\theta+\nu)^{\frac{1-q}{2}}\theta^{q-1}\,d\theta\right]\,\dx\,\dt.
\end{aligned}
\end{equation}

Now, we focus on the diffusion term $H_2$. For all
$(x,y,t)\in\R^N\times\R^N\times(0,\tau)$, set
\[
G(x,y,t):=u(x,t)-u(y,t),
\]
\[
F_{h, \eps}(x,y,t):=
t^{\frac12} \psi_{\varepsilon}(t)
\Big[
\bigl(u_{\bar h}(x,t)+\nu\bigr)^{\frac{1-q}{2}}\xi^2(x)
-
\bigl(u_{\bar h}(y,t)+\nu\bigr)^{\frac{1-q}{2}}\xi^2(y)
\Big],
\]
\[
F_{\eps}(x,y,t):=
t^{\frac12} \psi_{\varepsilon}(t)
\Big[
\bigl(u(x,t)+\nu\bigr)^{\frac{1-q}{2}}\xi^2(x)
-
\bigl(u(y,t)+\nu\bigr)^{\frac{1-q}{2}}\xi^2(y)
\Big].
\]
and
\[
F(x,y,t):=
t^{\frac12}
\Big[
\bigl(u(x,t)+\nu\bigr)^{\frac{1-q}{2}}\xi^2(x)
-
\bigl(u(y,t)+\nu\bigr)^{\frac{1-q}{2}}\xi^2(y)
\Big].
\]
First note that by Propositon \ref{con} \ref{con4}
\begin{equation}\label{mol3}
\int_0^\tau\iint_{B\times B}
\frac{|G(x,y,t)|^2}{|x-y|^{N+2s}}
\,\dx\,\dy\,\dt
\leq
\|u\|_{L^2(0,\tau;W^{s,2}(B))}^2.
\end{equation}
Now, we want to prove
\begin{equation} \label{ae}
   \lim_{h\to 0}F_{h, \eps}(x,y,t) = F_{\eps}(x,y,t), 
\end{equation}
for almost every $(x,y,t)\in\R^N\times\R^N\times(0,\tau)$.
To pass to the limit as $h\to0^+$, we consider the difference
\[
\int_0^\tau \iint_{\R^N\times\R^N}
G(x,y,t)\bigl[F_{h, \eps}(x,y,t)-F_{\eps}(x,y,t)\bigr]
K(x,y,t)\,\dx\,\dy\,\dt.
\]
Exploiting symmetry and recalling that $\xi=0$ in $\hat B^c$, we split this term as follows:
\begin{align} \label{h6h7}
&\int_0^\tau \iint_{\R^N\times\R^N}
G(x,y,t)\bigl[F_{h, \eps}(x,y,t)-F_{\eps}(x,y,t)\bigr]
K(x,y,t)\,\dx\,\dy\,\dt
\notag \\
={}&
\int_0^\tau \iint_{B\times B}
G(x,y,t)\bigl[F_{h, \eps}(x,y,t)-F_{\eps}(x,y,t)\bigr]
K(x,y,t)\,\dx\,\dy\,\dt
\notag \\
&+
2\int_0^\tau t^{\frac12} \psi_{\varepsilon}(t)
\iint_{\hat B\times B^c}
G(x,y,t)
\Big[
\bigl(u_{\bar h}(x,t)+\nu\bigr)^{\frac{1-q}{2}}
-
\bigl(u(x,t)+\nu\bigr)^{\frac{1-q}{2}}
\Big]
\xi^2(x)K(x,y,t)\,\dx\,\dy\,\dt
\notag \\
=:{}&H_6+H_7.
\end{align}
We show separately that both $H_6$ and $H_7$ tend to $0$ as $h\to0^+$, up to a subsequence.
We first deal with $H_6$. Adding and subtracting the term
\[
\bigl(u_{\bar h}(y,t)+\nu\bigr)^{\frac{1-q}{2}}\xi^2(x),
\]
using the fact that $0\leq\xi\leq1$ in $\R^N$ and \eqref{k}, we obtain
\begin{align} \label{ig}
&\int_0^\tau \iint_{B\times B}|F_{h, \epsilon}(x,y,t)|^2
K(x,y,t)\,\dx\,\dy\,\dt \notag\\
& \qquad
\leq
\gamma\int_0^\tau t \psi_{\varepsilon}(t) \iint_{B\times B}
\frac{
\left|
\bigl(u_{\bar h}(x,t)+\nu\bigr)^{\frac{1-q}{2}}
-
\bigl(u_{\bar h}(y,t)+\nu\bigr)^{\frac{1-q}{2}}
\right|^2
}{
|x-y|^{N+2s}
}
\,\dx\,\dy\,\dt \notag
\\
& \qquad \quad+
\gamma\int_0^\tau t \psi_{\varepsilon}(t) \iint_{B\times B}
\left(u_{\bar h}(y,t)+\nu\right)^{1-q}
\frac{
|\xi^2(x)-\xi^2(y)|^2
}{
|x-y|^{N+2s}
}
\,\dx\,\dy\,\dt \notag
\\
& \qquad =:H_8+H_9.
\end{align}
To estimate $H_8$, we observe that the map
\[
r\mapsto (r+\nu)^{\frac{1-q}{2}}
\]
is Lipschitz continuous in $[0,+\infty)$. Indeed, since $u_{\bar h}\geq0$, the mean value theorem yields
\[
\left|
\bigl(u_{\bar h}(x,t)+\nu\bigr)^{\frac{1-q}{2}}
-
\bigl(u_{\bar h}(y,t)+\nu\bigr)^{\frac{1-q}{2}}
\right|
\leq
\frac{q-1}{2}\nu^{-\frac{q+1}{2}}
\left|u_{\bar h}(x,t)-u_{\bar h}(y,t)\right|.
\]
Therefore, using $t \leq \tau$, $\psi_{\varepsilon}(t) \leq 1$, the inequality above and Proposition \ref{con} \ref{con4}, we obtain
\begin{align*}
H_8
\leq
\frac{\gamma \tau}{\nu^{q+1}}
\int_0^\tau \iint_{B\times B}
\frac{
\left|u_{\bar h}(x,t)-u_{\bar h}(y,t)\right|^2
}{
|x-y|^{N+2s}
}
\,\dx\,\dy\,\dt
\leq
\frac{\gamma \tau}{\nu^{q+1}}
\|u\|_{L^2(0,T_*;W^{s,2}(B))}^2.
\end{align*}
To estimate $H_9$, we use the inequalities
\[
|\xi^2(x)-\xi^2(y)|
\leq 2|\xi(x)-\xi(y)|
\leq
\frac{4}{(\sigma'-\sigma)\rho}|x-y|
\]
and
\[
\bigl(u_{\bar h}(y,t)+\nu\bigr)^{1-q}
\leq
\frac{\bigl(u_{\bar h}(y,t)+\nu\bigr)^2}{\nu^{q+1}}.
\]
Therefore, recalling that $t \leq \tau$, $\psi_{\varepsilon}(t) \leq 1$ and using
Proposition \ref{con} \ref{con2} with $r=2$, we obtain
\begin{align*}
H_9
&\leq
\frac{\gamma \tau}{(\sigma'-\sigma)^2\rho^2\nu^{q+1}}
\int_0^\tau \iint_{B\times B}
\bigl(u_{\bar h}(y,t)+\nu\bigr)^2
\frac{1}{|x-y|^{N+2s-2}}
\,\dx\,\dy\,\dt
\\
&\leq
\frac{\gamma \tau}{(\sigma'-\sigma)^2\rho^2\nu^{q+1}}
\left[
\int_0^\tau \int_B
\bigl(u_{\bar h}(y,t)+\nu\bigr)^2
\,\dy\,\dt
\right]
\left[
\sup_{y\in B}
\int_B
\frac{\dx}{|x-y|^{N+2s-2}}
\right]
\\
&\leq
\frac{\gamma \tau}{(\sigma'-\sigma)^2\rho^2\nu^{q+1}}
\left[
\int_0^\tau \int_B
u_{\bar h}^2(y,t)\,\dy\,\dt
+
\nu^2\rho^N \tau
\right]
\left[
\int_{B_{2\rho}(0)}
\frac{dz}{|z|^{N+2s-2}}
\right]
\\
&\leq
\frac{\gamma \tau}{(\sigma'-\sigma)^2\rho^{2s}\nu^{q+1}}
\left[
\|u\|_{L^2(B\times(0,T_*))}^2
+
\nu^2\rho^N \tau
\right].
\end{align*}
Combining the estimates for $H_8$ and $H_9$, we obtain from \eqref{ig} that $(F_{h,\varepsilon})_h$ is bounded in
\[
L^2\bigl(B\times B\times(0,\tau);K(x,y,t)\,\dx\,\dy\,\dt\bigr).
\]
By \eqref{ae}, up to a subsequence,
\[
F_{h, \eps} \rightharpoonup F_{\eps}
\quad\text{weakly in}\quad
L^2\bigl(B\times B\times(0,\tau);K(x,y,t)\,\dx\,\dy\,\dt\bigr).
\]
Moreover, by \eqref{mol3} and \eqref{k}, $G$ belongs to the same space. Therefore, up to a subsequence, the term $H_6$ in \eqref{h6h7} satisfies
\begin{equation}\label{mol5}
H_6\longrightarrow0
\qquad\text{as }h\to0^+.
\end{equation}
Next, we focus on $H_7$ in \eqref{h6h7}. First note that, for every
$x\in\hat B$ and $y\in B^c$, we have
\begin{equation}\label{ineq:distance_H7}
|x-y|
\geq |y|-|x|
\geq (1-\sigma')|y|.
\end{equation}
Recalling that $0\leq\xi\leq1$, $t \leq \tau$, $\psi_{\varepsilon}(t) \leq 1$ using \eqref{ineq:distance_H7}, the upper bound in \eqref{k}, and the inequality
\[
|G(x,y,t)|
\leq u(x,t)+|u(y,t)|,
\]
we obtain
\begin{align}\label{ineq:estimate_H7}
|H_7|
\leq{}&
\frac{\gamma \tau^{\frac12}}{(1-\sigma')^{N+2s}}
\int_0^\tau \iint_{\hat B\times B^c}
\frac{u(x,t)}{|y|^{N+2s}}
\left|
\bigl(u_{\bar h}(x,t)+\nu\bigr)^{\frac{1-q}{2}}
-
\bigl(u(x,t)+\nu\bigr)^{\frac{1-q}{2}}
\right|
\,\dx\,\dy\,\dt
\nonumber\\
&+
\frac{\gamma \tau^{\frac12}}{(1-\sigma')^{N+2s}}
\int_0^\tau \iint_{\hat B\times B^c}
\frac{|u(y,t)|}{|y|^{N+2s}}
\left|
\bigl(u_{\bar h}(x,t)+\nu\bigr)^{\frac{1-q}{2}}
-
\bigl(u(x,t)+\nu\bigr)^{\frac{1-q}{2}}
\right|
\,\dx\,\dy\,\dt
\nonumber\\
=:{}&
\frac{\gamma \tau^{\frac12}}{(1-\sigma')^{N+2s}}
\bigl(H_{10}+H_{11}\bigr).
\end{align}
We now estimate $H_{10}$. By the mean value theorem, we have
\begin{equation}\label{ineq:lagrange_H10}
\left|
\bigl(u_{\bar h}(x,t)+\nu\bigr)^{\frac{1-q}{2}}
-
\bigl(u(x,t)+\nu\bigr)^{\frac{1-q}{2}}
\right|
\leq
\frac{q-1}{2\nu^{\frac{q+1}{2}}}
\left|u_{\bar h}(x,t)-u(x,t)\right|.
\end{equation}
Therefore, applying H\"older's inequality with exponents $(2,2)$ and using \eqref{ineq:lagrange_H10}, we obtain
\begin{align}\label{ineq:estimate_H10}
H_{10}
&\leq
\left[
\int_0^\tau \int_{\hat B}u^2(x,t)\,\dx\,\dt
\right]^{\frac12}
\left[
\int_0^\tau \int_{\hat B}
\left|
\bigl(u_{\bar h}(x,t)+\nu\bigr)^{\frac{1-q}{2}}
-
\bigl(u(x,t)+\nu\bigr)^{\frac{1-q}{2}}
\right|^2
\,\dx\,\dt
\right]^{\frac12}
\int_{B^c}\frac{\dy}{|y|^{N+2s}}
\nonumber\\
&\leq
\frac{\gamma}{\nu^{\frac{q+1}{2}}\rho^{2s}}
\|u\|_{L^2(\hat B\times(0,\tau))}
\|u_{\bar h}-u\|_{L^2(\hat B\times(0,\tau))}.
\end{align}
Consequently, by Proposition \ref{con} \ref{con2} with $r=2$, we have
\[
H_{10}\longrightarrow0
\qquad\text{as }h\to0^+.
\]
As regards $H_{11}$, by \eqref{ineq:lagrange_H10}, we obtain
\begin{align}\label{ineq:estimate_H11}
H_{11}
&\leq
\frac{\gamma}{\nu^{\frac{q+1}{2}}}
\int_0^\tau
\left[
\int_{\hat B}
|u_{\bar h}(x,t)-u(x,t)|\,\dx 
\right]
\left[
\int_{B^c}
\frac{|u(y,t)|}{|y|^{N+2s}}\,\dy 
\right]
\,\dt
\nonumber\\
&\leq
\frac{\gamma}{\nu^{\frac{q+1}{2}}}
\left[
\sup_{t\in[0,\tau]}
\|u_{\bar h}(\cdot,t)-u(\cdot,t)\|_{L^1(\hat B)}
\right]
\textnormal{Tail}\bigl(|u|;B_\rho\times(0,\tau]\bigr).
\end{align}
By \cite[Lemma 2.6]{FDF} and \eqref{ineq:L_1_tail_condition}, it follows from
\eqref{ineq:estimate_H11} that
\[
H_{11}\longrightarrow0
\qquad\text{as }h\to0^+.
\]
Therefore, $H_7$ defined in \eqref{h6h7} satisfies
\begin{equation}\label{mol6}
H_7\longrightarrow0
\qquad\text{as }h\to0^+.
\end{equation}
Combining \eqref{mol5} and \eqref{mol6}, we obtain in \eqref{h6h7}, up to a subsequence,
\begin{equation} \label{h2h2h2}
    \int_0^\tau \iint_{\R^N\times\R^N}
G(x,y,t)\bigl[F_{h, \eps}(x,y,t)-F_{\eps}(x,y,t)\bigr]
K(x,y,t)\,\dx\,\dy\,\dt
\longrightarrow0,
\end{equation}
as $h\to0^+$. Hence, recalling $H_2$ in \eqref{mol1}, from \eqref{h2h2h2}, we get
\begin{equation}\label{mol7}
\lim_{h\to0^+}H_2
=
\int_0^\tau \iint_{\R^N\times\R^N}
G(x,y,t)F_{\eps}(x,y,t)K(x,y,t)\,\dx\,\dy\,\dt,
\end{equation}
up to a subsequence. Letting $\eps\to0^+$ and using the dominated convergence theorem, we obtain
\begin{align} \label{h2h2}
\lim_{\eps\to0^+} \lim_{h\to0^+} H_2
={}&
\int_{0}^{\tau}
\iint_{\R^N\times\R^N}
G(x,y,t)F(x,y,t)K(x,y,t)
\,\dx\,\dy\,\dt.
\end{align}
Finally, combining \eqref{mol2} and \eqref{h2h2} and passing to the limit in \eqref{mol1}, we conclude the proof.
\end{proof}

\end{document}